\documentclass[reqno]{amsart}
\usepackage{comment}
\usepackage[a4paper,hmargin=2cm,vmargin=3cm]{geometry}
\usepackage[T1]{fontenc}
\usepackage[utf8]{inputenc}
\usepackage{color}
\usepackage{amsmath}
\usepackage{amssymb}
\usepackage{mathtools}
\usepackage{bm}
\usepackage{bbm}
\usepackage[shortlabels]{enumitem}
\usepackage{hyperref}
\usepackage[normalem]{ulem}
\usepackage{bbm,enumitem}
\usepackage{epstopdf}
\usepackage{stmaryrd}
\usepackage{esint}
\usepackage{mathrsfs} 
\usepackage[english]{babel}
\usepackage{amssymb,enumerate}
\usepackage{color}
\usepackage{amsfonts}
\usepackage{amsthm}
\usepackage{graphicx,xcolor}
\usepackage{subfigure}
\usepackage{afterpage}
\makeatletter

\definecolor{notefontcolor}{rgb}{0.800781, 0.800781, 0.800781}
\definecolor{grey30}{rgb}{0.7,0.7,0.7}
\renewcommand{\eqref}[1]{\textup{\tagform@{\ref{#1}}}}

\numberwithin{equation}{section}

\theoremstyle{plain}
\newtheorem{theorem}{Theorem}[section]
\newtheorem{lemma}[theorem]{Lemma}
\newtheorem{proposition}[theorem]{Proposition}
\newtheorem{corollary}[theorem]{Corollary}

\theoremstyle{remark}
\newtheorem{remark}[theorem]{Remark}

\newcommand{\N}{{\mathbb{N}}}

\renewcommand{\Im}{\textup{Im }}

\newcommand{\e}{\e}

\newcommand{\R}{{\mathbb{R}}}
\renewcommand{\P}{{\mathbb P}}
\newcommand{\E}{\mathbb E}

\renewcommand {\(}{\left(}
\renewcommand {\[}{\left[}
\renewcommand {\)}{\right)}
\renewcommand {\]}{\right]}
\def\aled#1{\begin{aligned}#1\end{aligned}}

\def\be{\begin{equation}}
\def\ee{\end{equation}}
\def\bea{\begin{eqnarray}}
\def\eea{\end{eqnarray}}

\def\been#1{\begin{equation}#1\end{equation}}

\DeclareMathSymbol{\leqlant}{\mathalpha}{AMSa}{"36} 
\DeclareMathSymbol{\geqslant}{\mathalpha}{AMSa}{"3E} 
\DeclareMathSymbol{\eset}{\mathalpha}{AMSb}{"3F}     
\renewcommand{\leq}{\:\leqlant\:}                   
\renewcommand{\geq}{\:\geqslant\:}                   

\def\eqd{\stackrel{d}{=}}

\def\nn{\nonumber}
\def\a{\alpha}
\def\e{\varepsilon}
\def\d{\delta}

\def\b{\beta}

\def\l{\lambda}

\def\s{\sigma}

\def\R{\mathbb{R}}
\def\C{\mathbb{C}}

\def\k{\kappa}

\DeclareMathOperator{\atanh}{atanh}

\DeclareMathOperator{\card}{card}
\DeclareMathOperator{\ent}{Ent}

\DeclareMathOperator{\TAP}{TAP}
\DeclareMathOperator{\On}{On}

\DeclareMathOperator{\diag}{diag}
\DeclareMathOperator{\Tr}{Tr}
\DeclareMathOperator{\rank}{rank}
\DeclareMathOperator{\Var}{Var}

\DeclareMathOperator{\Span}{span}

\definecolor{light}{gray}{.9}

\definecolor{mypink1}{rgb}{0.858, 0.188, 0.478}
\definecolor{mygreen1}{rgb}{0.0, 0.8, 0.0}

\newcommand{\FTAP}{F_{\rm{TAP}}}

\newif\ifgeneralxi
\generalxifalse

\newcommand{\generalxi}[1]{%
  \ifgeneralxi
    \textcolor{blue}{#1}
  \fi
}

\title[Determinant in the Bray-Moore formula]{On the determinant in Bray-Moore's TAP complexity formula}

\author{David Belius} 
\address{Faculty of Mathematics and Computer Science, UniDistance Suisse, 3900 Brig, Switzerland}
\email{david.belius@cantab.net}

\author{Francesco Concetti}
\address{Faculty of Mathematics and Computer Science, UniDistance Suisse, 3900 Brig, Switzerland}
\email{francesco.concetti@unidistance.ch}

\author{Giuseppe Genovese}
\email{giuseppe.genovese@math.uzh.ch}

\date{\today}    
\begin{document}
\maketitle

\begin{abstract}
    In the computation of the TAP complexity, originally carried out by Bray and Moore, a fundamental step is to calculate the determinant of a random Hessian. As the replica method does not give a clear prescription, physicists debated how to perform this computation and its consequences on the TAP complexity for a long time. In this paper we prove the original Bray and Moore formula for the behaviour of the determinant at exponential scale to be correct, and compute an important prefactor coming from a small outlier in the spectrum.

\

\end{abstract}

\section{Introduction}\label{sect:intro}
    
    The problem of determining  asymptotics of absolute values of random determinants at exponential scale is receiving increasing attention in the last years, especially in connection to the computation of the complexity of random landscapes \cite{ABAC, fyodorovHighDimensionalRandomFields2013, benmon, Mon1, jos1, jos2, ben, sellke}. In this paper we investigate a specific random determinant appearing in a seminal article of Bray and Moore \cite{braymoore}, which we call the Bray-Moore determinant. The relevance of the Bray-Moore determinant is due not only to its fundamental importance for the Sherrrington-Kirkpatrick model, but also because its study fostered the development of many ideas in the study of high-dimensional non-convex functions \cite{fyo, benfyo, ros}.

    Bray and Moore \cite{braymoore} computed the number of metastable states of the Sherrington-Kirkpatrick model using the Kac-Rice formula \cite{AW}. Metastable states are understood as the critical point of the TAP free energy (after Thouless, Palmer, Anderson \cite{TAP}), that is
    \be\label{eq:FTAP}
        F_{\TAP}(m):=\frac{\b}{\sqrt N}\sum_{i,j \in [N]} J_{ij}m_im_j+h\sum_{i=1}^N m_i+\ent(m)+\frac{N\b^2}{2}(1-Q(m))^2\,,\qquad m\in[-1,1]^N\,.
    \ee
    \generalxi{
    $$
        F_{\TAP}(m):=\b H_{N}\left(m\right)+h\sum_{i=1}^N m_i+\ent(m)+N \On(Q(m))\,
    $$
    where
    $$
        \On(q):=\frac{\beta^2}{2}\left(\xi(1)-\xi'(q)(1-q)-\xi(q)\right).
    $$
    }
    Here $\b>0, h\in\R$, $J=\{J_{ij}\}_{i,j\in[N]}$ denotes a symmetric matrix with centred Gaussian entries: we take i.i.d. diagonal entries with $\mathbb{E}[J^2_{ii}]=2$, $i\in[N]$ and i.i.d. upper triangular entries with $\mathbb{E}[J_{ij}^2]=1$, $i,j\in[N]$; moreover $\ent(m)$ is the sum of the coin tossing entropy computed in the coordinates of $m\in[-1,1]^N$ and 
    \be\label{eq:random_Q}
        Q(m):=\frac{\|m\|_2^2}{N}\,. 
    \ee
    The Kac-Rice formula then gives for any measurable $B\subset(-1,1)^N\setminus\{0\}$
    \be\label{eq:TAP-KR}
    \mathbb{E}[\card\{m\in B\,:\,\nabla \FTAP(m)=0\}]=\int\!\!f(m)\mathbb{E}\left[\left|\det(\nabla^2 F_{\TAP})(m)\right|\,|\,\,\,\nabla F_{\TAP}(m)=0\right]dm\,,
    \ee
    where $f(m)$ is the density of the vector $\nabla \FTAP$ computed in zero (whose explicit form is not important here).  
    
    One crucial point in \cite{braymoore} is the computation of the expected random determinant of the Hessian appearing in \eqref{eq:TAP-KR}. This is the Bray-Moore determinant.
    The computation is rather delicate, even for $h=0$. Indeed in general computing the expected value of the modulus of the determinant by the replica method is a genuine technical challenge \cite{fyo}. Therefore different methods have been proposed to compute the determinant without the modulus and to obtain from this an estimate for the quantity with modulus. This question generated a debate in the theoretical physics community that lasted for over twenty years and which we briefly summarize below, referring to e.g. \cite{braymoore,plefka,kurchan,potters,cavagna,crisanti,braymoore2,rizzo,fyo} for more details.

    The gradient of the TAP free energy appearing in \eqref{eq:TAP-KR} is
    \begin{equation}
        \nabla F_{\TAP}(m)=\frac{\b}{\sqrt N}Jm+h-\atanh(m)-2\b^2m(1-Q(m))\,,\label{eq:nablaH-expl}
    \end{equation}
    \generalxi{
        \begin{equation}
            \nabla F_{\TAP}(m) = \nabla H_N(m)+h-\atanh(m)-\b^2m(1-Q(m))\xi''(Q(m))\,,\label{eq:nablaH-expl}
        \end{equation}
    }
    and the Hessian is
    \begin{equation}\label{eq:nabla2H-expl}
        \nabla^2 F_{\TAP}(m)
        =
        \frac{\b J}{\sqrt N}
        -
        D_N(m)
        +
        4\frac{\b^2}{N}mm^T\,,
        \end{equation}
    \generalxi{
        \begin{equation}
            \nabla^2 F_{\TAP}(m)
            =
            \nabla^2 H_N(m)
            -
            D_N(m)
            +
            \frac{2\b^2}{N}(\xi''(Q(m))-(1-Q(m))\xi'''(q))mm^T\,,\label{eq:nabla2H-expl}
        \end{equation}
    }   
    where
    \be\label{eq:defD}
        D_N(m):=\diag(\{(1-m_i^2)^{-1}+2\b^2(1-Q(m))\}_{i\in[N]})\,.
    \ee
    \generalxi{
        \be\label{eq:defD}
            D_N(m)
            :=
            \diag(
                \{(1-m_i^2)^{-1}
                +
                \b^2(1-Q(m))\xi''(Q(m))\}_{i\in[N]}
            )\,. 
        \ee
    }
    We study the determinant of \eqref{eq:nabla2H-expl} conditioned on the TAP equations $\nabla F_{\TAP}(m)=0$ being satisfied. 
    
    In \cite{braymoore} Bray and Moore computed (for $h=0$) the r.h.s. of (\ref{eq:TAP-KR}) neglecting the rank-one projector in \eqref{eq:nabla2H-expl} and arguing that the remaining term of the Hessian matrix is positive definite, which justifies dropping the absolute value. Identifying so $|\det|$ with $\det$, they found the following variational formula (see \cite[equation (13)]{braymoore})
    \bea
        &&\lim_{N \to \infty} \frac1N\log \mathbb{E}\left[\left|\det(\nabla^2 F_{\TAP})\right|\,|\,\,\,\nabla F_{\TAP}(m)=0\right]\label{eq:laformula}\\
        &=&\min_{g\leq 1-Q(m)} \left(\b^2g^2+\frac1N\sum^N_{i=1}\log((1-m_i^2)^{-1}
        +
        2\b^2(1-Q(m))-2\b^2g)\right)
        =:
        \Upsilon(\b,m)\,\label{eq:defUps}\,.
    \eea
    \generalxi{
        \bea
            &&\lim_N \frac1N\log \mathbb{E}\left[\left|\det(\nabla^2 F_{\TAP})\right|\,|\,\,\,\nabla F_{\TAP}(m)=0\right]\label{eq:laformula}\\
            &=&\min_{g\leq 1-Q(m)} \left(
                \frac{\b^2\xi''(Q(m))}{2}
                    g^2+\frac1N\sum^N_{i=1}\log((1-m_i^2)^{-1}
                +
                \b^2(1-Q(m))\xi''(Q(m))
                -
                \b^2\xi''(Q(m))g)
            \right)
            =
            :\Upsilon(\b,m)\,\label{eq:defUps}\,.
        \eea
    }
    They then assumed $g=1-Q(m)$ for the minimiser above and moved on with the complexity computation.

    In this paper we extend recent mathematically rigorous techniques for computing such expectations of determinants {\emph {with}} the absolute value \cite{ABAC,fyodorovHighDimensionalRandomFields2013,subag2017complexity,ben,Mon1} to derive a rigorous proof of the formula of Bray-Moore.
    We furthermore compute a correction term, which is subleading for most $m$, 
    but becomes large when $m$ is close to zero. In a future work we will use the estimates proved here to compute the complexity of TAP solutions from \eqref{eq:TAP-KR} mathematically rigorously, and controlling the correction term is necessary to obtain a result for $m$ close to zero when $h=0$. The correction term arises because of an outlier eigenvalue of the Hessian which is close to zero. Unlike the leading order term, it is affected by the rank-one projector in \eqref{eq:nabla2H-expl}.

    In the physics literature Plefka \cite[Section 3]{plefka} and later Kurchan \cite{kurchan} pointed out that the rank-one projector term in \eqref{eq:nabla2H-expl} cannot be overlooked a priori. Kurchan \cite{kurchan} noted an apparent contradiction arising from Morse theory (see for instance \cite[page 100-101]{bott}), which implies that the quantity 
    \begin{equation}\label{eq:morse}
        \sum_{m^* \text{ critical points of $F_{\TAP}$} } (-1)^{\#\text{ positive eigenvalues of }\nabla^2 F_{\TAP}(m^*)}
    \end{equation}    
    is a topological invariant of the space $[-1,1]^N$ (called the Euler-Poincar\'e characteristic). Its value for $\beta$ small - and therefore for all $\beta$ due to invariance - is expected to be one. At the same time, for $\beta$ large one expects the sum \eqref{eq:morse} to have an exponential number of terms, which seems hard to reconcile with it equalling one. Kurchan suggested that the issue may be resolved by the presence of a subleading prefactor. Later Aspelmeier, Bray and Moore \cite{braymoore2} interpreted this as the appearance of critical points in pairs so as to cause cancellations in in \eqref{eq:morse}. They related this to the outlier eigenvalue of the Hessian, and to a small prefactor in the complexity calculation. Our work is a mathematically rigorous confirmation of part of this analysis.
        
We now state the main result. The correction term is formulated in terms of the quantity (recall \eqref{eq:random_Q}, \eqref{eq:defD})
    \be\label{eq:defu-v}
        v=v(m):=\atanh m-h+2\b^2m(1+Q(m))-D_N(m)m\,.
    \ee    
    
    \begin{theorem}\label{thm:MainDet}
        For all $\a>0$ there exists a constant $c:=c(\a)$ such that
        \be\label{eq:MainDet}
            \left|
                \frac1N\log \mathbb{E}\left[|\det\nabla^2 F_{\TAP}(m)|\,\big|\, \nabla \FTAP(m)=0 \right]
                -
                \Upsilon(\b,m)
                -
                \frac1N\log\(\frac{\|v\|^2_2}{\|m\|^2_2}+\frac{|(m,v)|}{\|m\|^2_2}-\frac{|(m,v)|^2}{\|m\|^4_2}\)
            \right|
            \leq
            \frac{c}{N^{\frac{1}{120}}}\,,
        \ee
        for all $N\ge 1$ and all 
        $m \in (-1,1)^N$
        satisfying
        \be\label{eq:cutoff-m}
            \max_{i \in [N]}|m_i| \le 1 - e^{-\a\sqrt N}
        \ee        
        and
        \begin{equation}\label{eq:v_cond}
            \frac{\|v\|^2_2}{\|m\|^2_2}+\frac{|(m,v)|}{\|m\|^2_2}-\frac{|(m,v)|^2}{\|m\|^4_2}>0
        \end{equation}
        (cf. the second $\log$ term in \eqref{eq:MainDet}),
        where $\Upsilon$ is as in \eqref{eq:defUps} and $v$ is as in \eqref{eq:defu-v}.        
        
        Furthermore, for each $N\ge 1$  there are finitely many $m\in(-1,1)^N$ for which \eqref{eq:v_cond} does not hold, and for these $m$ the l.h.s. of \eqref{eq:v_cond} equals zero and $\mathbb{P}\left[ \det(\nabla^2 \FTAP(m)) = 0 \,\big|\, \nabla \FTAP(m)=0 \right]=1$.
        
        Lastly, for $m \in (-1,1)^N \setminus\{0\}$ not satisfying \eqref{eq:cutoff-m}
        \been{
            \label{eq:outside.TH}
            \aled{
                \frac1N\log\E[
                    |\det\nabla^2 F_{\TAP}(m)| 
                    \big|
                    \nabla \FTAP(m)=0  ]&\leq \frac1N\log\(\frac{\|v\|^2_2}{\|m\|^2_2}+\frac{|(m,v)|}{\|m\|^2_2}-\frac{|(m,v)|^2}{\|m\|^4_2}\)\\
                &+2\log\(16\b(1+\b^2)\)-\frac{17}{N}\sum_{i\in[N]}\log(1-m^2_i).
        }}
    \end{theorem}
 
    In our future application of Theorem \ref{thm:MainDet} to computing the complexity of TAP solutions - already mentioned above - the sharpness exhibited by the presence of the correction term in \eqref{eq:MainDet}, the rate on the r.h.s. of \eqref{eq:MainDet}, the condition \eqref{eq:cutoff-m} and the estimate \eqref{eq:outside.TH} for $m$-s that violate the condition will be necessary to obtain a result that covers all of $(-1,1)^N$.

    The analysis of the variational principle in (\ref{eq:defUps}) yields two possible choices for the minimisers, namely $g=1-Q$ or $g=g^*\in(0,1-Q)$ (this was already observed in \cite{Mon1}). The first one corresponds in the Bray-Moore notations to setting $B=0$ (first display at page L473 of \cite{braymoore}). We will show that imposing the Plefka condition \cite{plefka} 
    \be\label{eq:Plefkaset0}
    2\b^2\sum_{i=1}^N(1-m_i^2)^2\leq N\,
    \ee
    forces the minimiser to be in fact $1-Q$. 
    
    Among the recent contributions on the topic of asymptotics of random determinants, the analyses of \cite{Mon1} and \cite{ben} greatly inspired this work. 
    In \cite{Mon1} the first mathematical results on the Bray-Moore random determinant have been achieved. Indeed the authors obtain an upper bound for \eqref{eq:laformula} valid for $Q(m)$ bounded away from zero (a regime in which the correction given by the isolated eigenvalue is negligible) and corresponding to formula (\ref{eq:defUps}) evaluated in $g=1-Q$ (see \cite[Proposition 3.2]{Mon1}).
    In \cite{ben} a very systematic study of the behaviour of absolute value of determinants at exponential scale of a wide class of random matrices has been carried out. However, the Bray-Moore determinant studied here does not satisfy the assumptions of \cite{ben}. More precisely, two crucial problems prevent us from applying the result of \cite{ben} as a black-box: the operator norm of $D_N(m)$ in (\ref{eq:defD}) grows to infinity as $m$ nears the boundary of $[-1,1]^N$, and the presence of a single outlier in the spectrum which can be arbitrarily close to zero independently of $N$. These two features prevent the informal assumption (1) or the assumption (C) of \cite{ben} from holding, and pose new technical challenges that we address in this paper. 
    
    As an intermediate step towards the proof of our main Theorem \ref{thm:MainDet}, we prove the following asymptotics for random matrices of the form $\text{GOE}+D$ where the operator norm of $D$ is large, in the sense that \cite[equation (1.4)]{ben} is violated, but the spectrum is sufficiently separated from zero, that is \cite[equation (1.5)]{ben} is satisfied (compare also with \cite[Remark 1.4]{ben} and \cite{ben2}). We also include a (large) lower rank term $A$ and show that its presence does not affect the asymptotics of the determininat. 
    
    In what follows $\|D\|_{\rm op}$ denotes the operator norm of $D$ and $\mu\boxplus\nu$ denotes the additive convolution between the probability measure $\nu$ and the semi-circle law $\mu_{{\rm sc},\b}$, defined by
    \be\label{eq:semicerchio}
    \mu_{{\rm sc},\b}(dx):=\bm{1}_{[-2\b, 2\b]}\frac{1}{2\pi}\sqrt{4 \b^2-x^2}\,.
    \ee
    
    \begin{theorem}
    \label{thm:GOE_plus_D_det}
    Let $N\in\N$, $\b>0$, $J$ be a $N\times N$ GOE with $\E[J^2_{ij}]=1+\d_{ij}$ for all $i\leq j$. Let $A,D\in \R^{N\times N}$ be deterministic $N\times N$ real symmetric matrices, such that 
    \been{
    \|A\|_{op}, \|D\|_{\rm op}\leq N^{1-\d}\,
    }
    for some $\d>0$ and $\rank(A)\leq p\leq N/2$. 
    Then there exists some $c:=c(\b)$ such that
    \been{\label{eq:GOE_plus_D_det}
        \left|
            \frac{1}{N}\log\mathbb{E}\left[\left|\det\left(\frac{\b}{\sqrt{N}}J+A+D\right)\right|\right]
            -
            \int\log|x|(\nu_{D}\boxplus\mu_{{\rm sc},\b})(dx)
        \right|
        \leq
            \frac{c(\b)}{N^{\min\(\d,\frac{1}{60}\)}}
            +
            \frac{pc(\b)}{N^{\min\(\d,\frac{1}{2}\)}}
    }
    where $\nu_D$ is the empirical spectral distribution of $D$.
    \end{theorem}
    
    The rest of the paper is organised as follows. In Section \ref{sect:GOE} we prove a number of useful properties of the Stieltjes transform and of the additive convolution with the semicircle law (\ref{eq:semicerchio}) of a probability measure. The proof of Theorem \ref{thm:GOE_plus_D_det} is presented in Section \ref{section:RegDet} and Section \ref{sect:logdet}: in the first one we focus on the case in which the lower rank matrix $A$ is absent, while the effect of this contribution is evaluated in Section \ref{sect:logdet}. In Section \ref{sect:prel} we show that the Bray-Moore formula (\ref{eq:laformula}) is directly related to the additive convolution of the spectral measure of $D_N$ with the semicircle law. This reduces the problem of computing the Bray-Moore determinant to proving Theorem \ref{thm:MainDet2} below, in which the additive convolution appears in the asymptotics of the log-determinant (similarly as in Theorem \ref{thm:GOE_plus_D_det}). Theorem \ref{thm:MainDet2} is proven in Section \ref{sect:noN}. The proof is based on a convenient representation of the Bray-Moore determinant, which on a very high level can be written as 
    $$
    \sum_{i}\k_i(m)\det({\rm GOE}+D_N(m)+\text{rank-1 matrix}_i)1_{\mathcal P_i}(m),
    $$
    where $\{\mathcal P_i\}$ is a suitable partition of $(-1,1)^N$ to which we remove a neighbourhood of the boundary, see \eqref{eq:cutoff-m}, and $\k_i\in\R$. The crucial point is that the rank-1 matrices in the representation above are deterministic, so we are able to recover the asymptotics of the determinants by Theorem \ref{thm:GOE_plus_D_det}. The factors $\k_i(m)$ are explicitly computed and they are responsible for the correction to the Bray-Moore variational formula in (\ref{eq:MainDet}).
        
    \subsection{Notation}
    We use the standard notation $[N]:=\{1,\ldots, N\}$ for $N\in\N$. 
    Everywhere $(\cdot,\cdot)$ denotes the Euclidean inner product in $\R^N$ and $\|x\|_2^2:=(x,x)$. 
    We write complex numbers $z=u+iv$, with $u,v\in \R$. $\C^+$ is the complex half-plane with strictly positive imaginary part.
    
    Throughout we call a GOE matrix $J=\{J_{ij}\}_{i,j\in[N]}$ a doubly indexed sequence of centred Gaussian random variables, independent modulo $J_{ij}=J_{ji}$ and normalised such that $\mathbb{E}[J^2_{ii}]=2$, $i\in[N]$ and $\mathbb{E}[J_{ij}^2]=1$ for $i<j$.
    
    The identity matrix on $\R^N$ is denoted by $I_N$. 
    Given $N\in \N$ and $i\in[N]$, we denote by $e^i_N\in \R^N$ the standard basis vectors:
    \been{
    \label{eq:def_eM}
    (e^i_N)_{k}:=\d_{ik},\quad \forall k\in [M].
    }
    If $E$ is a linear subspace of $\R^N$, we let $P_E$ denote the associated orthogonal projector and $P_E^\perp:=I_N-P_E$. For $v_1,v_2\in\R^N$ we shorten $P_{v_1}:=P_{\Span\{v_1\}}$ and $P_{v_1,v_2}:=P_{\Span\{v_1,v_2\}}$. Moreover we set $P_{[k]}:=P_{\Span\{e^1_N\ldots e^k_N\}}$.

     Given a $N\times N$ symmetric matrix $M$, let $$\lambda_1(M)\geq \lambda_2(M)\geq \cdots\geq\lambda_N(M)$$ be the eigenvalues of $M$ sorted in increasing order. Also, we define the least eigenvalue
    \been{\label{eq:lleast}
    \l_{\min}(M):=\min_j\{|\lambda_j(M)|:\,j\in[N]\,\}.
    }
    For such a matrix $M$, we write $M^2:=M^*M$, where $M^*$ is the Hermitian conjugate of $M$. We write respectively the operator and Frobenius norm as
    $$
    \|M\|_{\rm op}:=\sup_{\|x\|_2=1}\|Mx\|_2\,\,,\qquad \|M\|_F:=\sqrt{\Tr M^2}.
    $$

    We denote by $M^{([p],[p])}$ be the $(N-p)\times(N-p)$ matrix obtained by removing the first $p$ rows and the first $p$ columns from the matrix $M$, i.e.:
    \been{
    \label{eq:def_Ap}
    (M^{([p],[p])})_{ij}:=M_{i+p,j+p},\quad \forall (i,j)\in [N-p]^2.
    }
    Let us also write $M^{(1,1)}=M^{([1],[1])}$.
    
    The empirical distribution of the eigenvalues of $M$ is
    \been{
        \label{eq:nu_a}
        \nu_M:=\frac{1}{N}\sum^N_{j=1}\delta_{\lambda_j(D)}.
    } 
    Often we will shorten
    \be\label{eq:defZ}
        Z:=\frac{\b}{\sqrt N}J+D,
    \ee
    where $\b>0$, $J$ is a GOE matrix and $D$ a constant matrix. In Section \ref{sect:prel} and Section \ref{sect:noN} we will write $D$ for the matrix $D_N(m)$ defined in \ref{eq:defD}.
    
    The semicircle law with variance $\sqrt2\b>0$ defined by (\ref{eq:semicerchio}) is always denoted by $\mu_{{\rm sc},\b}$. The additive (or free) convolution of the probability measures $\mu,\nu$ is $\mu\boxplus\nu$. The Stieltjes transform of the probability measure $\nu$ is always denoted by $\widehat\nu$. The precise definitions of these objects are given in Section \ref{sect:GOE}.

    $C,c$ everywhere denote positive absolute constants whose value may change from formula to formula. We write $X\lesssim Y$ if $X\leq CY$ and $X\simeq Y$ if $Y\lesssim X\lesssim Y$. For any discrete set $A$, $\card A$ denotes its cardinality and for any Borel set $B$, $|B|$ denotes its Lebesgue measure.

\subsection*{Acknowledgements} The authors are grateful to P. Bourgade, S. Franz, B. McKenna, F. Ricci-Tersenghi, B. Schlein and K. Schnelli for helpful discussions. D.B. and F. C. are supported by the SNSF grants 176918 and 206148.


\section{Stieltjes transform and additive convolution}\label{sect:GOE}

In this section we collect a number of results about Stieltjes transforms and additive convolutions of probability measures that are used in the sequel. We start by recalling the definitions.
Let $\nu$ be a probability measure on $\R$. Its Stieltjes transform always indicated here by $\widehat\nu$ is defined by
\be\label{eq:stile10}
\widehat \nu(z):=\int \frac{\nu(dx)}{z-x},\quad \forall z\in \C^+\,.
\ee

The additive convolution of $\nu$ with the semicircle law $\mu:=\nu \boxplus \mu_{\rm{sc}, \sigma}$ is the unique probability measure whose Stieltjes transform satisfies
\been{
\label{eq:stile20}
\widehat{\mu} (z) =\int_{\R} \frac{\nu(dx)}{x-\sqrt{2}z+2\b^2\widehat{\mu} (z)},\quad \forall z\in \C^+.
}
We list several useful properties of the Stieltjes transform in the next lemma.
\begin{lemma}
Let $\nu$ be a probability measure on $\R$ and denote by $\widehat{\nu}$ its Stieltjes transform. It holds:
\begin{itemize}
    \item for all $z\in \C^+$:
    \been{
    \label{eq:posIm}
    \textup{Im}(\widehat{\nu}(z))> 0,
    }
    \been{
    \label{eq:ope1}
    |\widehat{\nu}(z)|\leq \frac{1}{\textup{Im}(z)},
    }
    and, for any $p\geq 0$:
\been{
\label{eq:oper}
    \int\frac{\nu(dx)}{|x-z|^{2+p}}\leq \frac{\textup{Im}(\nu(z))}{(\textup{Im}(z))^{p+1}},
    }
    with the equality if $p=0$;
 \item for any analytic function $f: \C^+\to \C^+$ and $z\in \C^+$:
 \been{
  \label{eq:ope2}
    |\widehat{\nu}(z+f(z))|\leq \frac{1}{\textup{Im}(z)};
    }
    \item given $c>0$, if $|f(z)|\in (0,c |\textup{Im}(z)|^{-1})$, then it holds:
\been{
 \label{eq:ope3}
\int \frac{\nu(dx)}{|x-z-f(z)|^2}\leq \frac{c}{c+\(\textup{Im}(z)\)^2}\(\frac{\textup{Im}(\widehat{\nu}(z+f(z)))}{\textup{Im}(f(z))}\)
}
\item for any $v>0$:
\been{
 \label{eq:opeint}
\int_{\R} \textup{Im}(\widehat{\nu}(u+iv))du=\pi.
}
\end{itemize}
\end{lemma}
\begin{proof}
Let $x\in \R$ and $z\in \C^+$. We have
\been{
\label{eq:ope0000}
\frac{1}{|x-z|}\leq \frac{1}{\textup{Im}(z)},
}
and
\been{
\label{eq:ope000}
\textup{Im}\(\frac{1}{x-z}\)=\frac{\textup{Im}(z)}{|x-z|^2}.
}
From the identity \eqref{eq:ope000}, we get
\been{
\label{eq:ope0}
\textup{Im}(\widehat{\nu}(z))=\int \textup{Im}\(\frac{1}{x-z}\)\nu(dx)=\textup{Im}(z)\int \frac{\nu(dx)}{|x-z|^2}.
}
So if $z\in \C^+$ then $\textup{Im}(\widehat{\nu}(z))>0$, proving \eqref{eq:posIm}. Given $p\geq 0$, the inequality \eqref{eq:ope0000} and the identity \eqref{eq:ope0} give
\been{
\int\frac{\nu(dx)}{|x-z|^{2+p}}\leq \frac{1}{(\textup{Im}(z))^p}\int\frac{\nu(dx)}{|x-z|^{2}}=\frac{1}{(\textup{Im}(z))^{p+1}}\(\textup{Im}(z)\int \frac{\nu(dx)}{|x-z|^2}\)=\frac{\textup{Im}(\widehat{\nu}(z))}{(\textup{Im}(z))^{p+1}},
}
that proves \eqref{eq:oper}. Using the triangular inequality and \eqref{eq:ope0000} we get
\been{
\label{eq:ope00}
|\widehat{\nu}(z)|\leq\int\frac{\nu(dx)}{|x-z|}\leq \int \frac{\nu(dx)}{|\textup{Im}(x-z)|}=\frac{1}{|\textup{Im}(z)|}\int \nu(dx)=\frac{1}{|\textup{Im}(z)|}
}
which proves the upper bound  \eqref{eq:ope1}. Let us now consider any analytic function $f:\C^+\to \C^+$. Since $\textup{Im}(f(z))> 0$, then
\been{
|\textup{Im}(f(z)+z)|=|\textup{Im}(f(z))+\textup{Im}(z)|\geq \textup{Im}(z),\quad \forall z\in \C^+.
}
Thus, from \eqref{eq:ope00}, we have
\been{
|\widehat{\nu}(z+f(z))|\leq \frac{1}{|\textup{Im}(f(z)+z)|}\leq \frac{1}{\textup{Im}(f(z))+\textup{Im}(z)}\leq \frac{1}{\textup{Im}(z)},
}
proving \eqref{eq:ope2}. Inverting \eqref{eq:ope0}, we get
\been{
\int \frac{\nu(dx)}{|x-z-f(z)|^2}=\frac{\textup{Im}(\widehat{\nu}(z+f(z)))}{\textup{Im}(f(z))+\textup{Im}(z)}.
}

If $|f(z))|\leq c\,\textup{Im}(z)^{-1}$, then $\textup{Im}(f(z))\in (0,c\,\textup{Im}(z)^{-1})$ and the above formula gives
\been{
\label{eq:imineq2}
\int \frac{\nu(dx)}{|x-z-f(z)|^2}=\frac{c}{c+c\frac{\textup{Im}(z)}{\textup{Im}(f(z))}}\(\frac{\textup{Im}(\widehat{\nu}(z+f(z)))}{\textup{Im}(f(z))}\)\leq \frac{c}{c+\(\textup{Im}(z)\)^2}\(\frac{\textup{Im}(\widehat{\nu}(z+f(z)))}{\textup{Im}(f(z))}\)
}
and this proves \eqref{eq:ope3}.

Finally, we prove the integration formula \eqref{eq:opeint}:
\been{
\label{eq:Stile_int}
\int_{\R}\textup{Im}(\widehat{\nu}(u+iv))du
=\int_{\R}\nu(dx)\(\int_{\R}\frac{vdu}{(x-u)^2+v^2}\)
\leq \int_{\R}\nu(dx)\(\int_{\R}\frac{dy}{y^2+1}\)= \pi,
}
where, in the first equality, we used the identity \eqref{eq:ope000} and the Fubini Theorem, and in the last equality we use the integral formula $\int_{\R}(x^2+1)^{-1}dx=\pi$ and $\int_{\R}\nu(dx)=1$.
\end{proof}

Consider now a GOE $N\times N$ matrix $J$ and a deterministic $N\times N$ symmetric matrix $D$. Set for $\b>0$
\been{
\label{eq:Zf}
Z:=\frac{\b}{\sqrt{N}}J+D
}
We now study the $L_p$ convergence of the Stieltjes transform of the empirical spectral measure of $Z$ to the Stieltjes transform of $\nu_{D}\boxplus\mu_{{\rm sc},\b}$. To this end we generalise the approach of \cite[Lemma 4.6]{Mon1}.

Let $z=u+i v$, with $u\in \R$ and $v>0$. We define
\been{\label{eq:G_stile}
G_Z(z):=(Z-z\,I_N )^{-1}\,,\qquad G(z):=(D-(z+2\b^2 \E\[\widehat{\nu}_Z(z)\])I_N)^{-1}. 
}
Without loss of generality we assume in the remaining of this section that $D$ is diagonal. This in turn implies that also $G$ is diagonal.

Next we give two accessory lemmas.
\begin{lemma}\label{lemma:verde}
The following bounds hold true:
\bea
\max_{i\in[N]}|G_{ii}(z)|&\leq& \frac1v\,, \label{upper_G}\\
\|G\|_F&\leq& \frac{\sqrt{N}}{v}\label{upper_GHS}\\
\E[\|G_Z\|^2_F]&\leq& \frac{N\E\[\textup{Im}(\widehat\nu_Z(z))\]}{v}\,, \label{upper_GZHS-1}\\
\E[\|G_Z^2\|^2_F]&\leq& \frac{N\E\[\textup{Im}(\widehat\nu_Z(z))\]}{v^3} \label{upper_GZHS-2}\,.
\eea
\end{lemma}
\begin{proof}
We have
\bea
\max_{i\in[N]}|G_{ii}(z)|&=&\max_{i\in[N]}\frac{1}{|D_{ii}-z-2\b^2 \E\[\widehat{\nu}_Z(z)\]|}\nn\\
&\leq& \max_{i\in[N]}\frac{1}{|\textup{Im}(D_{ii}-z-2\b^2 \E\[\widehat{\nu}_Z(z)\])|}\nn\\
&=&\frac{1}{v+2\b^2\E\[\Im \widehat{\nu}_Z(z)\]}\leq \frac1v.
\eea
which proves \eqref{upper_G}. 

The bound \eqref{upper_GHS} follows readily from \eqref{upper_G}.

Now we prove \eqref{upper_GZHS-1}, \eqref{upper_GZHS-2}. Let $q=1,2$. By definition of the empirical measure $\nu_Z$ we have
\been{
\label{eq:GZIm}
\aled{
&\frac{1}{N}\E\[\|G^q_Z(z)\|^2_F\]=\E\[\frac{1}{N}\textup{Tr}((Z-zI_N)^{-q}(Z-\overline{z}I_N)^{-q})\]\\
&=\E\[\(\frac{1}{N}\sum^N_{i=1}\frac{1}{|\l_{i}(Z)-z|^{2q}}\)\]\leq \E\[\frac{1}{N}\sum^N_{i=1}\frac{1}{|\l_{i}(Z)-z|^{2q}}\]\\
&=\E\[\int_{\R}\frac{\nu_Z(dx)}{|x-z|^{2q}}\]\leq\frac{\E\[\textup{Im}(\widehat\nu_Z(z))\]}{v^{2q-1}}.
}
}
Here we applied the Jensen inequality in the first bound, \eqref{eq:oper} in the second one.
\end{proof}

\begin{lemma}\label{lemma:poincare}
It holds
\be\label{eq:poincare}
\Var\[\widehat{\nu}_Z(z)\]\leq \frac{2\b^2\E\[\textup{Im}(\widehat\nu_Z(z))\]}{N^2v^3}\,. 
\ee
\end{lemma}
\begin{proof}
Noting that
$$
\widehat{\nu}_Z(z)=\frac1N\textup{Tr}\(G_Z(z)\),
$$
we define $\nabla \widehat{\nu}_Z(z)$ as a $N\times N$ matrix with elements
\be
(\nabla \widehat{\nu}_Z(z))_{ij}:=\partial_{J_{ij}} \widehat{\nu}_Z(z)=\frac1N\sum_{k=1}^N\partial_{J_{ij}}\frac{1}{Z_{kk}-z}=-\frac{\sqrt2\b}{N^{\frac32}}\sum_{k=1}^N\frac{\d_{ik}\d_{jk}}{(Z_{kk}-z)^2}\,.
\ee
Therefore
\bea
\|(\nabla \widehat{\nu}_Z(z))\|_F^2&=&\sum_{ij}(\nabla \widehat{\nu}_Z(z))_{ij}^2
=\frac{2\b^2}{N^3}\sum_{i,j,h,k=1\ldots N}\frac{\d_{ik}\d_{jk}\d_{ih}\d_{jh}}{(Z_{kk}-z)^2(Z_{hh}-z)^2}\nn\\
&=&\frac{2\b^2}{N^3}\sum_{i,j=1\ldots N}\frac{\d_{ij}}{(Z_{ij}-z)^4}=\frac{2\b^2}{N^3}\|G_Z^2(z)\|_F^2\,. 
\eea
Thus by (\ref{upper_GZHS-2})
\be\label{eq:nabla}
\|(\nabla \widehat{\nu}_Z(z))\|_F^2\leq \frac{2\b^2\E\[\textup{Im}(\widehat\nu_Z(z))\]}{N^2v^3}\,.
\ee
We recover \eqref{eq:poincare} combining the bound (\ref{eq:nabla}) and the Poincaré inequality for the Gaussian measure
$$
\Var\[\widehat{\nu}_Z(z)\]\leq E[\|\nabla \widehat{\nu}_Z(z)\|_F^2].
$$
\end{proof}

\begin{lemma}\label{lem:convergence}
Let $N\in\N$, $\b>0$, $J$ be a $N\times N$ GOE matrix, $D$ any $N\times N$ deterministic matrix and $Z$ as in \eqref{eq:Zf}.
Let (recall \eqref{eq:semicerchio} and \eqref{eq:nu_a})
\been{
\label{eq:concentration0mu}
\mu:=\nu_{D}\boxplus\mu_{{\rm sc},\b}.
}

Assuming that
\be\label{eqNlarge}
N\geq 16\frac{\b^6}{v^6}\(\frac{2\b^2+3v^2}{8\b^2+v^2}\),
\ee
it holds for any $p\geq1$
\been{
\label{eq:Lp_conv}
\(\int_{\R} |\widehat{\mu}(u+iv)-\E[\widehat{\nu}_Z(u+iv)]|^pdu\)^{\frac{1}{p}}\leq \frac{\pi^{\frac1p}2\b^2(2\b^2+v^2)(2\b^2+3v^2)}{Nv^{7-\frac1p}}
}
and
\been{
\sup_{u\in\R}|\widehat{\mu}(u+iv)-\E[\widehat{\nu}_Z(u+iv)]|\leq \frac{2\b^2(2\b^2+v^2)(2\b^2+3v^2)}{Nv^{7}}.
}
\end{lemma}

\begin{proof}

We define
\been{\label{eq:def-nuhat}
g(z):=\int_{\R}\frac{\nu_D(dx)}{x-2\b^2\E[\widehat{\nu}_Z(z)]-z}\,
}
and
\been{
\label{eq:nuZnu}
r(z):=\E\[\widehat{\nu}_Z(z)\]-g(z)\,.
}

We split (recall $z=u+iv$) for $p\in[1,\infty]$
\bea
\(\int_{\R}|\widehat{\mu}(z)-\E[\widehat{\nu}_Z(z)]|^pdu\)^{\frac{1}{p}}&\leq&\(\int_{\R}|r(u+iv)|^pdu\)^{\frac{1}{p}}\label{eq:termine1}\\
&+&\(\int_{\R}|\widehat{\mu}(z)-g(z)|^pdu\)^{\frac{1}{p}}.\label{eq:termine2} 
\eea

We bound the term on the r.h.s. of \eqref{eq:termine1}. Following \cite[Lemma 4.6, Equation (4.8)]{Mon1} we write
\been{
r(z)=\frac{2\b^2}{N^2}\textup{Tr}(\E[G_Z(z)^2]G(z))+\frac{2\b^2}{N}\textup{Tr}(\E[\(\widehat{\nu}_Z(z)-\E\[\widehat{\nu}_Z(z)\]\)G_Z(z)G(z))]=:\textup{I}+\textup{II}
}
($\textup{I}$ being the first addendum and $\textup{II}$ the second one). 
It is
\been{
 \label{eq:term1}
 \aled{
| \textup{I}|&\leq \frac{2\b^2}{N^2}\sum^N_{i=1}\sum^N_{j=1}|(G_Z(z))_{ij}|^2|(G(z))_{jj}|\\
&\leq \frac{2\b^2}{N^2}\max_{j\in [N]}|(G(z))_{jj}| \sum^N_{i=1}\sum^N_{j=1}\E[|(G_Z(z))_{ij}|^2]\\
&\leq  \frac{2\b^2}{N^2v}\E\[\|G_Z(z)\|^2_F\]\leq  \frac{2\b^2\E\[\textup{Im}(\widehat\nu_Z(z))\]}{Nv^2}\,, 
}
}
where we used \eqref{upper_G} in the penultimate inequality and \eqref{upper_GZHS-1} in the last one. 

Moreover by the Cauchy-Schwarz inequality we have
\been{
\label{eq:term2}
\aled{
 |\textup{II}| &\leq\frac{2\b^2}{N}\sqrt{\E\[|\textup{Tr}\(G_Z(z)G(z)\)|^2\]}\sqrt{\E\[|\widehat{\nu}_Z(z)-\E\[\widehat{\nu}_Z(z)\]|^2\]}\\
&\leq \frac{2\b^2}{N}\|G(z)\|_F\sqrt{\E\[ \|G_Z(z)\|^2_F\]}\sqrt{\Var\[\widehat{\nu}_Z(z)\]}\\
&= \frac{2\b^2}{N^{\frac{5}{2}}}\|G(z)\|_F\sqrt{\E\[ \|G_Z(z)\|^2_F\]}\sqrt{N^3\Var\[\widehat{\nu}_Z(z)\]}\\
&\leq \frac{2\b^2}{2N^2v}\(\E\[ \|G_Z(z)\|^2_F\]+N^3\Var\[\widehat{\nu}_Z(z)\]\)\\
&\leq \frac{2\b^2}{2Nv^2}\(1+\frac{2\b^2}{v^2}\),
}
}
where we the third inequality is (\ref{upper_GHS}) and in the last step we used the bounds (\ref{upper_GZHS-1}) and (\ref{eq:poincare}). Combining \eqref{eq:term1} and \eqref{eq:term2} we get
\be\label{eq:Brr}
|r(z)|\leq\frac{2\b^2}{Nv^2} \(\frac32+\frac{2\b^2}{2v^2}\)\E\[\textup{Im}(\widehat\nu_Z(z))\].
\ee
We can now estimate the r.h.s. of \eqref{eq:termine1} for any $p\geq1$ as
\bea
\(\int |r(u+iv)|^pdu\)^{\frac1p}&\leq& \frac{2\b^2}{Nv^2}\(\frac32+\frac{2\b^2}{2v^2}\)\(\int \E\[\textup{Im}(\widehat\nu_Z(z))\]^pdu\)^{\frac1p}\nn\\
&\leq& \frac{2\b^2}{Nv^{3-\frac1p}}\(\frac32+\frac{2\b^2}{2v^2}\) \(\int \E\[\textup{Im}(\widehat\nu_Z(z))\]du\)^{\frac1p}\nn\\
&\leq & \frac{\pi^{\frac1p}2\b^2}{Nv^{3-\frac1p}}\(\frac32+\frac{2\b^2}{2v^2}\).\label{eq:bound-r-p}
\eea
where we used the H\"older inequality and the bounds (\ref{eq:ope1}), (\ref{eq:opeint}). The estimate (\ref{eq:bound-r-p}) holds true also at $p=\infty$, as it can be seen by plugging the bound (\ref{eq:ope1}) into \eqref{eq:term1} and \eqref{eq:term2}:
\be\label{eq:bound-r-infty}
\sup_{u\in\R}|r(u+iv)|\leq \frac{2\b^2}{Nv^{3}}\(\frac32+\frac{2\b^2}{2v^2}\).
\ee

Now we study the term in (\ref{eq:termine2}). Bearing in mind the definition (\ref{eq:def-nuhat}) and that
\be\label{eq:selfish}
\widehat\mu= \int\frac{\nu_D(dx)}{x-2\b^2\widehat\mu(z)-z}
\ee
we write
\been{
\label{eq:munu}
\aled{
&|\widehat{\mu}(z)-g(z)|\leq  \int_\R\left|\frac{1}{x-2\b^2\E[\widehat{\nu}_Z(z)]-z}-\frac{1}{x-2\b^2\widehat{\mu}(z)-z}\right|\nu_D(dx)\\
& =2\b^2|\widehat{\mu}(z)-\E[\widehat{\nu}_Z(z)]|\int_{\R}\frac{\nu_D(dx)}{|x-2\b^2\E[\widehat{\nu}_Z(z)]-z||x-2\b^2\widehat{\mu}(z)-z|}\\
&\leq 2\b^2|\widehat{\mu}(z)-\E[\widehat{\nu}_Z(z)]|\sqrt{\int_{\R}\frac{\nu_D(dx)}{|x-2\b^2\E[\widehat{\nu}_Z(z)]-z|^2}}\sqrt{\int_{\R}\frac{\nu_D(dx)}{|x-2\b^2\widehat{\mu}_D(z)-z|^2}}
}
}
by the Cauchy-Schwarz inequality.

By \eqref{eq:ope2} it is $|\widehat{\mu}(z)|\leq v^{-1}$. Thus the \eqref{eq:ope3} applies with $c=1$ and we have
\been{\label{eq:lassi}
\int_{\R}\frac{\nu_D(dx)}{|x-2\b^2\widehat{\mu}(z)-z|^2}\leq \frac{1}{2\b^2+v^2}\(\frac{\textup{Im}\(\widehat{\nu}_D(z+\widehat{\mu}(z))\)}{\textup{Im}\(\widehat{\mu}(z)\)}\)=\frac{1}{2\b^2+v^2}.
}
The last identity follows from \eqref{eq:selfish}. Similarly we have
\been{
\label{eq:lassa}
\aled{
\int_{\R}\frac{\nu_D(dx)}{|x-2\b^2\E[\widehat{\nu}_Z(z)]-z|^2}&\leq
\frac{1}{2\b^2+v^2}\(\frac{\textup{Im}\(\widehat\nu_D(z+\E[\widehat{\nu}_Z(z)])\)}{\textup{Im}(\E[\widehat{\nu}_Z(z)])}\)\\
&=\frac{1}{2\b^2+v^2}\(\frac{\textup{Im}\(g(z)\)}{\textup{Im}(\E[\widehat{\nu}_Z(z)])}\)\\
&\leq \frac{1}{2\b^2+v^2}\(1+\frac{|r(z)|}{\textup{Im}(\E[\widehat{\nu}_Z(z)])}\)\\
&\leq \frac{1}{2\b^2+v^2}\(1+\frac{2\b^2}{Nv^2} \(\frac32+\frac{2\b^2}{2v^2}\)\)
}
}
where we used (\ref{eq:Brr}) in the last bound. Combining (\ref{eq:munu}), (\ref{eq:lassi}) and (\ref{eq:lassa}) 
we get
\be
|\widehat{\mu}(z)-g(z)|\leq \(\frac{2\b^2}{2\b^2+v^2}\sqrt{1+\frac{2\b^2}{Nv^2} \(\frac32+\frac{2\b^2}{2v^2}\)}\)|\widehat{\mu}(z)-\E[\widehat{\nu}_Z(z)]|\,.
\ee
A direct computation shows that if (\ref{eqNlarge}) holds then
$$
\(\frac{2\b^2}{2\b^2+v^2}\sqrt{1+\frac{2\b^2}{Nv^2} \(\frac32+\frac{2\b^2}{2v^2}\)}\)\leq \frac12\(1+\frac{2\b^2}{v^2+2\b^2}\).
$$
Hence
\be
\(\int_\R|\widehat{\mu}(u+iv)-g(u+iv)|^pdu\)^{\frac1p}\leq \frac12\(1+\frac{2\b^2}{v^2+2\b^2}\) \(\int_\R|\widehat{\mu}(u+iv)-\E[\widehat{\nu}_Z(u+iv)]|^pdu\)^{\frac1p}\,. 
\ee
Therefore by (\ref{eq:termine1}), (\ref{eq:termine2}) we have for $p\in[1,\infty]$
\be\label{eq:fineprova}
\(\int_\R|\widehat{\mu}(u+iv)-\E[\widehat{\nu}_Z(u+iv)]|^pdu\)^{\frac1p}\leq \frac{2(2\b^2+v^2)}{v^2}\(\int_{\R}|r(u+iv)|^pdu\)^{\frac{1}{p}}\,,
\ee
and by (\ref{eq:bound-r-p}) (for $p\in[1,\infty)$) and (\ref{eq:bound-r-infty}) (for $p=\infty$) the assertion is proven. 
\end{proof}

We close with the following property of the additive convolution with the semi-circle law. While this is a direct upshot of the analysis of \cite{biane}, it is worth to single it out for the sequel. 
\begin{lemma}
\label{lem:p_boundness}
Let $\nu$ be any probability measure on $\R$ for which the additive convolution $\mu=\nu\boxplus\mu_{{\rm sc},\b}$ is well defined. Then $\mu$ has a density function $p:\R\to [0,\infty)$ that is analytic in  $\{x\in  \R\,:\,p(x)>0\}$ and verifies
\been{\label{eq:unifBOund}
\sup_{x\in\R}p(x)\leq \frac{1}{\pi\sqrt2\b}.
}
\end{lemma}
\begin{proof}
The fact that the measure $\mu$ has a density $p$ that is analytic on the set $\{x\in  \R\,:\,p(x)>0\}$ is well known after \cite[Corollary $3$ and $4$]{biane}. The point here is to show the upper bound (\ref{eq:unifBOund}), which is independent on $\nu$.

Following \cite{biane} we introduce the function
\been{
\upsilon_\b(u):=\inf \left\{v>0\,:\,\int_\R\frac{\nu(dx)}{(u-x)^2+v^2}\leq\frac{1}{2\b^2}\right\}.
}
From \cite[Corollary $3$]{biane} it follows that
\been{
\sup_{x\in \R}p(x)\leq\sup_{x\in \R}\frac{\upsilon_\b(x)}{\pi 2\b^2}.
}
Since clearly $\upsilon_\b(u)\leq \b$ we recover the assertion. 
\end{proof}


\section{Proof of Theorem \ref{thm:GOE_plus_D_det} without the lower rank term}\label{section:RegDet}

In this section we prove Theorem \ref{thm:GOE_plus_D_det} assuming $p=0$, that is that the lower rank term $A$ vanishes. 

For any symmetric matrix $M\in\R^{N\times N}$ we shorten
\been{
\label{eq:LA}
L(M):=\frac{1}{N}\log(|\det(M)|).
}
Next we prove the formula \eqref{eq:GOE_plus_D_det} for the regularised determinant $L(Z+i\e):= L(Z+i\e I_N)$.

\begin{proposition}
\label{lem:concentration0}
Let $N\in\N$ (large enough) and $\e\in[N^{-\frac{1}{10}},N^{-1/12}]$. Let also $\b>0$, $J$ be a $N\times N$ GOE matrix, $D$ a $N\times N$ deterministic matrix and $Z$ as in \eqref{eq:Zf}.
Let (recall \eqref{eq:semicerchio} and \eqref{eq:nu_a})
\been{
\label{eq:concentration0mu}
\mu:=\nu_{D}\boxplus\mu_{{\rm sc},\b}.
}
Assume that for some $\d\in (0,1]$
\been{
\|D\|_{\textup{op}}\leq e^{N^{1-\d}}.
}
Then for any $r>0$ it holds
\been{
\label{eq:concentration20}
\left|\frac{1}{Nr}\log\(\E\[e^{rNL(Z+i \e)}\]\)-\int\log|x|\mu(dx)\right|\leq\frac{c}{N^{\min\(\frac{1}{24},\d\)}}
}
\end{proposition}
\begin{proof}
The triangular inequality yields
\bea
&&\left|\frac{1}{N r}\log\(\mathbb{E}[e^{rN L(Z+i \e)}]\)- \int\log|x|\mu(dx)\right|\label{eq:montanari0}\\
&\leq& \left|\frac{1}{N r}\log\(\mathbb{E}[e^{rN L(Z+i \e)}]\)-\frac{1}{N} \mathbb{E}[L(Z+i \e)]\right|\label{eq:montanari00}\\
&+&\left|\frac{1}{N}\mathbb{E}[L(Z+i \e)]- \int\log|x+i\e|\mu(dx)\right|\label{eq:montanari01}\\
&+&\left|\int\log|x+i\e|\mu(dx)- \int\log|x|\mu(dx)\right|.\label{eq:montanari02}
\eea
For any $a>0$ the function $\log|a x+i\e|$ is Lipschitz with constant bounded by $a/\e$. Therefore \cite[Theorem 2.3.5]{AGZ} gives
\be\label{eq:concentration30}
\P\left(|L(Z+i \e)-\mathbb{E}[L(Z+i \e)]|\geq t\right)\leq 2e^{-\frac{\e^2}{8\b^2 }N^2t^2}\,. 
\ee
Hence, for $\e>0$,
\been{\aled{
\mathbb{E}[e^{rN (L(Z+i \e)-\mathbb{E}[L(Z+i \e)])}]&\leq \mathbb{E}[e^{rN |L(Z+i \e)-\mathbb{E}[L(Z+i \e)]|}]\\&=rN\int_0^\infty  e^{rtN}\P\left(\left|  L(Z+i \e)-\mathbb{E}[L(Z+i \e)]\right| \geq t\right)dt\\
&\leq rN\int_{-\infty}^\infty e^{rNt-\frac{\e^2}{8\b^2 }N^2t^2}dt \leq 2\sqrt{2\pi}\b \e^{-1}\,e^{2r^2\b^2 \e^{-2}}.
}}
Moreover by the Jensen inequality $\mathbb{E}[e^{rN (L(Z+i \e)}]\geq e^{rN\mathbb{E}[L(Z+i \e)]}$.
Therefore
\be\label{eq:concentation40}
\aled{
\eqref{eq:montanari00}&=\left|\frac{1}{N r}\log\(\mathbb{E}[e^{rN L(Z+i \e)}]\)-\frac{1}{N} \mathbb{E}[L(Z+i \e)]\right|\\
&\leq \left|\frac{1}{N r}\log\(\mathbb{E}[e^{rN (L(Z+i \e)-\mathbb{E}[L(Z+i \e)])}]\)\right|\\
&\leq  \left|\frac{2r\b^2}{N \e^2} +\frac{1}{2Nr}\log\(\frac{8\pi\b^2 }{\e^2}\)\right|\\
&\leq \frac{4r\b^2}{N\e^2}\lesssim \frac{r\b^2}{N^{\frac45}},
}
\ee
where the penultimate inequality holds for $\e$ small enough and in the last one we used $\e\geq N^{-1/10}$. 

We now estimate \eqref{eq:montanari01}. By \cite[Lemma 4.5]{Mon1}
\be\label{eq:montanari000}
\left|\mathbb{E}[L\(Z+i \e\)]- \int\log|x+i\e|\mu(dx)\right|\leq\frac{\log N}{N}+\frac{C}{N}\left(\frac{1}{\e^5}+\log(1+\|D\|_{\rm op})\right)\,. 
\ee
Since $\|D\|_{\rm op}\leq e^{N^{1-\d}}$ and again $\e\geq N^{-1/10}$, for $N$ large enough it holds
\be\label{eq:montanari0000}
\eqref{eq:montanari01}\leq\frac{\log N}{N}+\frac{C}{N\e^5}+\frac{C}{N^{\d}}\lesssim \frac{1}{N^{\min(\d,\frac12)}}.
\ee

It remains to deal with \eqref{eq:montanari02}. 
Using that $\log|x|\leq\log|x+i\e|\leq\log|x|+\frac\e{|x|}$, for $\e<1$ we get
\been{
\aled{
&|\log|x+i\e|-\log|x||=\log(|x+i\e|)-\log|x|\\
&\leq \frac{1}{2}\log(4\e)\bm{1}_{\{|x|\leq \sqrt{\e}\}}+\(\log|x|+\frac{\e}{|x|}\)\bm{1}_{\{|x|> \sqrt{\e}\}}-\log|x|\\
&\leq \frac{1}{2}\log(4\e)\bm{1}_{\{|x|\leq \sqrt{\e}\}}+\(\log|x|+\sqrt{\e}\)\bm{1}_{\{|x|> \sqrt{\e}\}}-\log|x|\\
&\leq \(\frac{1}{2}\log(4\e)-\log|x|\)\bm{1}_{\{|x|\leq \sqrt{\e}\}}+\sqrt{\e}.
}
}
Hence
\be\label{eq:upper_logD0}
\aled{
\eqref{eq:montanari02}&=\left|\int\log|x+i\e|\mu(dx)-\int\log|x|\mu(dx)\right|\\
&\leq \int^{\sqrt{\e}}_{-\sqrt{\e}}\(\frac{1}{2}\log(4\e)-\log|x|\)\mu(dx)+\sqrt{\e}\\
&\leq\frac{1}{\pi \s}\int^{\sqrt{\e}}_{-\sqrt{\e}}\(\frac{1}{2}\log(4\e)-\log|x|\)dx+\sqrt{\e}\\
&=\frac{1}{\pi \s}\(\sqrt{\e}\log(4\e)-2\sqrt{\e}\log(\sqrt{\e})+2\sqrt{\e}\)+\sqrt{\e}\lesssim \sqrt{\e},
}
\ee
where in the second inequality we used Lemma \ref{lem:p_boundness}. 

Combining the bounds \eqref{eq:concentation40}, \eqref{eq:montanari0000} and \eqref{eq:upper_logD0} we obtain
\be
\eqref{eq:montanari0}\lesssim \frac{r\b^2}{N^{\frac45}}+\frac{1}{N^{\min(\d,\frac12)}}+\sqrt\e\lesssim \frac{1}{N^{\min(\d,1/24)}}
\ee
since $\e\leq N^{-1/12}$.
\end{proof}

Ideally, to prove the Theorem \ref{thm:GOE_plus_D_det} we wish to replace $L(Z+i\e)$ with $L(Z)$ in formula \eqref{eq:concentration20}. However, taking $\e\to0$ requires some additional care. 

Recall (see \eqref{eq:lleast}) that $\l_{\min}:=\l_{\min}(Z)$ denotes the least eigenvalue of $Z$. We have 
\begin{lemma}
\label{lem:lminP}
Let $Z$ be a $N\times N$ symmetric matrices defined as in  Lemma \ref{lem:convergence}. Then
\been{
\label{eq:lminP}
\P\(\l_{\min}\leq  \frac{1}{t}\)\leq \frac{CN^2}{\b\,t}\,,
 }
for some universal constant $C>0$ independent on $N$.
\end{lemma}
\begin{proof}
See \cite[Lemma 6.2]{SST}.
\end{proof}
In the next proposition following the approach of \cite{ben} we control the difference between $|\det(Z)|$ and its regularisation at an exponential scale. 
\begin{proposition}\label{lemma:panini0}
Let $N\in\N$ (large enough) and $\e\in [N^{-\frac{1}{10}},N^{-\frac{1}{12}}]$. Let also $\b>0$, $J$ be a $N\times N$ GOE matrix, $D$ any $N\times N$ deterministic matrix and $Z$ as in \eqref{eq:Zf}.
Let (recall \eqref{eq:semicerchio} and \eqref{eq:nu_a})
\been{
\label{eq:concentration0mu}
\mu:=\nu_{D}\boxplus\mu_{{\rm sc},\b}.
}
Assume that for some $\d\in (0,1]$
\been{
\|D\|_{\textup{op}}\leq e^{N^{1-\d}}.
}
Then there exists a constant $c:=c(\b)>0$ such that
\be\label{eq:Lipdiff0}
\frac1N\log\left(1-\frac{c}{N^{1/60}}\right)-\frac{C}{N^{1/60}}+\frac1N\log\mathbb{E}[e^{N L(Z+i \e)}]\leq\frac1N\log\mathbb{E}[e^{N L(Z)}]\leq \frac1N\log\mathbb{E}[e^{N L(Z+i \e)}].
\ee
\end{proposition}

\begin{proof}
Clearly $L(Z)\leq L(Z+i \e)$ for any $\e\geq 0$ gives the upper bound. So we have to prove the lower bound. 
Everywhere in the proof we shorten $\l_i=\l_i(Z)$ and $\l_{\min}:=\l_{\min}(Z)$.
We set 
\be\label{eq:Def-W_qpk0}
W_{N,\e}:=\{\,\card\{i\in[N]\,:\,|\l_i|\leq \e^2\}\leq N^{\frac{19}{20}}\,,\quad \l_{\min}> e^{-N^{1/30}}\,\}.
\ee
We have
\been{\label{eq:banana0}\aled{
\mathbb{E}[e^{N L(Z)}]&=\mathbb{E}[e^{N (L(Z)-L(Z+i \e))}e^{N L(Z+i \e)}]\\
&=\mathbb{E}\[e^{-\frac{1}{2}\sum_{i\in[N]}\log\left(1+\frac{\e^2}{\l_i^2}\right)}e^{N L(Z+i \e)}\]\\
&\geq \mathbb{E}\[e^{-\frac{1}{2}\sum_{i\in[N]}\log\left(1+\frac{\e^2}{\l_i^2}\right)}e^{N L(Z+i \e)}1_{W_{N,\e}}\]\,.
}}

On the event $W_{N,\e}$ it holds
\been{\aled{
-\sum_{i\in[N]}\log\left(1+\frac{\e^2}{\l_i^2}\right)&=-\sum_{\substack{i\in[N]\\i\,:\,|\l_i|\leq \e^2 }}\log\left(1+\frac{\e^2}{\l_i^2}\right)
-\sum_{\substack{i\in[N]\\i\,:\,|\l_i|> \e^2}}\log\left(1+\frac{\e^2}{\l_i^2}\right)\\
&\geq -N^{\frac{19}{20}}\log\left(1+\e^2e^{2N^{\frac{1}{30}}}\right)-N\log\left(1+\e^{-2}\right)\\
&\gtrsim -N^{\frac{19}{20}+\frac{1}{30}}-N\e^{-2}\gtrsim -N^{\frac{59}{60}} 
}}
as $\e\gtrsim N^{-\frac{1}{10}}$.
By the Cauchy-Schwarz inequality
\been{\label{eq:chehai0}\aled{
\mathbb{E}[e^{N L(Z)}]&\geq e^{-N^{\frac{59}{60}}}\mathbb{E}[e^{NL(Z+i \e)}1_{W_{N,\e}}]\\
&=e^{-N^{\frac{59}{60}}}\(\mathbb{E}[e^{N L(Z+i \e)}]-\mathbb{E}[e^{N L(Z+i \e)}1_{W^c_{N,\e}}]\)\\
&\geq e^{-N^{\frac{59}{60}}}\(\mathbb{E}[e^{N L(Z+i \e)}]-\sqrt{\mathbb{E}[e^{2NL(Z+i \e)}]}\sqrt{\P(W^c_{N,\e})}\)\\
&=e^{-N^{\frac{59}{60}}}\mathbb{E}[e^{N L(Z+i \e)}]\left(1-\sqrt{\frac{\mathbb{E}[e^{2N L(Z+i \e)}]}{\mathbb{E}[e^{N L(Z+i \e)}]^2}}\sqrt{\P(W^c_{N,\e})}\right)\,.
}}
It follows from \eqref{eq:concentration20} (with $r=2$) that
\been{
\label{eq:concen_pan}
\frac{\mathbb{E}[e^{2N L(Z+i \e))}]}{\mathbb{E}[e^{N L(Z+i \e)}]^2}\leq e^{cN^{\max\left\{-\frac{1}{10},-\d\right\}}}.
}
It remains only to bound $\P(W^c_{N,\e})$. Clearly
\be\label{eq:decompP0}
\P(W^c_{N,\e})\leq \P\left(\l_{\min}\leq e^{-N^{\frac{1}{30}}}\right)+\P\left(\card\{i\,:\,|\l_i|\leq \e^2\}\geq N^{\frac{19}{10}}\right). 
\ee
We estimate these two addenda separately. The first one is readily controlled by Lemma \ref{lem:lminP}. We have
\be\label{eq:stimaP20}
\P\left(\l_{\min}\leq e^{-N^{\frac{1}{30}}}\right)\leq CN^2e^{-N^{\frac{1}{30}}}\,,
\ee
for some constant $C>0$ independent on $N$.

For the second term in \eqref{eq:decompP0} we use the Markov inequality
\be\label{eq:markov2}
\P\left(\card\{i\,:\,|\l_i|\leq \e^2\}\geq N^{\frac{19}{20}}\right)\leq N^{-\frac{19}{20}}\E\left[\card\{i\,:\,|\l_i|\leq \e^2\}\right]\,.
\ee
We have (bear in mind \eqref{eq:concentration0mu})
\bea
&&\E\left[\card\{i\,:\,|\l_i|\leq \e^2\}\right]=N\E\[\int_{-\e^2}^{\e^2}\nu_{Z}(dx)\]\nn\\
&=&N\int_{-\e^{2}}^{\e^2}\mu(dx)+N\int_{-\e^2}^{\e^2}(\E\[\nu_Z(dx)\]-\mu(dx)). \label{eq:sono30}
\eea
By Lemma \ref{lem:p_boundness} it holds
\be\label{eq:m30}
N\int_{-\e^2}^{\e^2}\mu(dx)\leq \frac{\sqrt2}{\pi\b}N\e^{2}\,. 
\ee
For the second addendum in \eqref{eq:sono30} we bound
\be\label{eq:prebai0}
\int_{-\e^2}^{\e^2}(\E\[\nu_Z(dx)\]-\mu(dx))\leq \sup_{x\in \R}|F_N(x)-F(x)|\,,
\ee
where
\been{
F_N(x):=\int_{-\infty}^x \mathbb{E}[\nu_Z(dx)]\,,\qquad F(x):=\int_{-\infty}^x \mu(dx)\,. 
}
By \cite[Theorem 2.1]{bai} we have
\begin{align}
\sup_{x\in \R}|F_N(x)-F(x)|&\leq C\(\int_{\R}|\E[\widehat\nu_Z(u+i\e)]-\widehat\mu(u+i\e)|du\label{eq:Bai10}\right.\\
&\left.+\frac 1\e\sup_{x}\int_{|y|\leq 4\e}|F(x+y)-F(x)|dy\)\,. \label{eq:Bai30}
\end{align}
The condition
\be\label{eq:condNv}
\e^6\geq C\frac{\b^6}{N}
\ee
yields by Lemma \ref{lem:convergence} (with $p=1$) 
\been{
\label{eq:Bai1bis0}
\int_{\R}|\E[\widehat\nu_Z(u+i\e)]-\widehat\mu(u+i\e)|du\lesssim \frac{\b^6}{N\e^6}.
}

Note that since $\e>N^{-\frac{1}{10}}$ \eqref{eq:condNv} is always satisfied. 

Now we estimate \eqref{eq:Bai30}. We write
\be
|F(x+y)-F(x)|=
\begin{cases}
\int_x^{x+y}\mu(dt)&y>0\\
\int^x_{x+y}\mu(dt)&y<0
\end{cases}\,.
\ee
Then, using Lemma \ref{lem:p_boundness}, we get
\be\label{eq:Bai3bis0}
\eqref{eq:Bai30}\leq \frac1\e\int_{|y|\leq 4\e}\left|\int_x^{x+y}\mu(dt)\right|dx\leq \frac{64}{\pi\b}\e\,. 
\ee
Combining \eqref{eq:prebai0}, \eqref{eq:Bai10}, \eqref{eq:Bai30}, \eqref{eq:Bai1bis0} and \eqref{eq:Bai3bis0} we find
\be\label{eq:prebai-fine}
\int_{-\e^2}^{\e^2}(\E\[\nu_Z(dx)\]-\mu(dx))\lesssim \frac\e\b+\frac{\b^2}{N\e^6}\,. 
\ee
This together with (\ref{eq:sono30}), (\ref{eq:m30}) gives
\be\label{eq:premarco0}
\mathbb{E}\left[\card\{i\,:\,|\l_i|\leq \e^2\}\right]\lesssim \frac{N\e}{\b}+\frac{\b^2}{\e^6}.\,
\ee
When we plug this bound into (\ref{eq:markov2}) we obtain
\be\label{eq:marco0}
\P\left(\card\{i\,:\,|\l_i|\leq \e^2\}\geq N^{\frac{19}{20}}\right)\lesssim\frac{N^{\frac{1}{20}}\e}{\b}+\frac{\b^2}{N^{\frac{19}{20}}\e^6}\lesssim\frac{1}{N^{\frac{1}{30}}} 
\ee
as $\e\in[N^{-1/10},N^{-1/12}]$.
Combining \eqref{eq:decompP0}, \eqref{eq:stimaP20}, \eqref{eq:marco0},  we get
\be\label{eq:stimaP22}
\P(W^c_{N,\e})\lesssim \frac1{N^{\frac{1}{30}}}\,.
\ee
Using this bound in \eqref{eq:banana0} together with \eqref{eq:concen_pan} ends the proof.
\end{proof}

\begin{proof}[Proof of Theorem \ref{thm:GOE_plus_D_det} (with $A=0$)]
Take any $\e\in[N^{-\frac{1}{10}},N^{-1/12}]$. We have
\been{
\label{eq:concentration200}
\aled{
&\left|\frac{1}{N}\log\(\E\[e^{NL(Z)}\]\)-\int\log(|x|)\mu(dx)\right|\\
&\leq \left|\frac{1}{N}\log\(\E\[e^{NL(Z+i \e)}\]\)-\int\log(|x|)\mu(dx)\right|\\
&+\left|\frac{1}{N}\log\(\E\[e^{NL(Z)}\]\)-\frac{1}{N}\log\(\E\[e^{NL(Z+i \e)}\]\)\right|\\
&\lesssim \frac{1}{N^{\min(1/24,\d)}}+\frac{1}{N^{1/60}}. 
}
}
We estimated in the last inequality the first summand by Proposition \ref{lem:concentration0} (with $r=1$) and the second one by Proposition \ref{lemma:panini0}. 
\end{proof}

\section{Log-determinants with a lower rank term}\label{sect:logdet}

In this section we show that the asymptotics of the determinant established in the previous section is stable if we add a lower rank term that does not need to be small. This will complete the proof of Theorem \ref{thm:GOE_plus_D_det}. Recall from Section \ref{section:RegDet} that for any $N\times N$ matrix $M$ we have
\be\label{eq:defLsect4}
L(M)=\frac1N\log\det M\,,\qquad L(M+i\e)= L(M+i\e I_N)\,. 
\ee

As in the last section we analyse first the regularised determinant.
The goal here is to prove the following extension of Proposition \ref{lem:concentration0}.

\begin{proposition}
\label{lem:concentration}
Let $\b>0$, $N,k,p\in\N$ with $k<N$ and $0\leq2p\leq N-k$. Consider 
\been{
\label{eq:def_Zstar_mat}
Z_*:=\frac{\b}{\sqrt{N}}J-(OD_N(m)O^T)^{([k],[k])}+A,
}
where $J$ is a $(N-k)\times (N-k)$ GOE matrix, $O$ is a $N\times N$ deterministic orthogonal matrix, $D,A$ are symmetric deterministic $(N-k)\times (N-k)$ matrices such that for a $\d\in(0,1)$ it holds
\be\label{eq:AD}
\|A\|_{op},\|D\|_{op}\leq e^{N^{1-\d}}.
\ee
Moreover $1\leq\rank A\leq p$. 
Then for any $\e>0$ there exists a constant $c:=c(\a,\b,k,p)>0$ such that (recall $\mu=\nu_D\boxplus \mu_{\rm{sc},\b}$)
\been{
\left|\frac{1}{Nr}\log\(\E\[e^{r NL(Z_*+i\e)}\]\)-\int_{\R}\log(|x|)\mu(dx)\right|\leq c\(\frac{p+k}{N^\d}+\frac{1}{N^{\min\(\frac{1}{24},\d\)}}\).
} 
\end{proposition}
The following corollary is the main result of this section. We obtain Theorem \ref{thm:GOE_plus_D_det} from Corollary \ref{lem:almost_fin_det} just taking $k=0$. 
\begin{corollary}
\label{lem:almost_fin_det}
Let $\b>0$, $N,k,p\in\N$ with $k<N$ and $0\leq2p\leq N-k$. Consider $Z_*$ as in (\ref{eq:def_Zstar_mat}), 
where $J$ is a $(N-k)\times (N-k)$ GOE matrix, $O$ is a $N\times N$ deterministic orthogonal matrix, $D,A$ are symmetric deterministic $(N-k)\times (N-k)$ matrices such that for a $\d\in(0,1)$ (\ref{eq:AD}) holds and $1\leq\rank A\leq p$.
Then
\been{
\left|\frac1N\log \mathbb{E}\left[|\det(Z_*)|\right]- \int \log|x|\mu(dx)\right|\lesssim \frac{1}{(N-k)^{\min\(\d,\frac{1}{60}\)}}+\frac{p+k}{N^\d}+\frac{1}{N^{\min\(\d,\frac{1}{24}\)}}.
}
\end{corollary}
\begin{proof}
We split
\bea
&&\left|\frac1N\log \mathbb{E}\left[|\det(Z_*)|\right]- \int \log|x|\mu(dx)\right|\nn\\
&\leq &\left|\frac1N\log \mathbb{E}\left[|\det(Z_*)|\right]- \frac1N\log \mathbb{E}\left[|\det(Z_*+i\e I_{N-k})|\right]\right|\label{eq:detAll1}\\
&+ &\left|\frac1N\log \mathbb{E}\left[|\det(Z_*+i\e I_{N-k})|\right]- \int \log|x|\mu(dx)\right|.\label{eq:detAll2}
\eea
Since $Z_*$ is a $N-k\times N-k$ GOE plus a constant matrix satisfying the bounds (\ref{eq:AD}), Proposition \ref{lemma:panini0} applies. Thus, for $\e\in (N^{-\frac{1}{10}},N^{-\frac{1}{12}})$ it holds
\been{
\label{eq:detI}
\eqref{eq:detAll1}\lesssim \frac{1}{(N-k)^{\min\(\d,\frac{1}{60}\)}}.
}
Moreover, Proposition \ref{lem:concentration} and again (\ref{eq:AD}) give
\been{
\label{eq:detII}
\eqref{eq:detAll2}\lesssim\frac{p+k}{N^\d}+\frac{1}{N^{\min\(\d,\frac{1}{24}\)}}.
}
Combining \eqref{eq:detI}, and \eqref{eq:detII} we end the proof.
\end{proof}

\

The proof of Proposition \ref{lem:concentration} needs some preparation. We start by two deterministic lemmas. The first one compares the quantity $L$ (see \eqref{eq:defLsect4}) between matrices which differ by a lower rank matrix. Then in the second lemma we deduce a similar bound for certain minors of a given matrix. 

\begin{lemma}\label{lemma:erstatz}
Let $k,N\in \N$ with $N>2k$. Let $A,B\in \R^{N\times N}$ be symmetric, invertible matrices such that $\textup{rank}(A-B)\leq k$.
Then
\be\label{eq:erstatz-compl}
\aled{
N|L(A)-L(B)|<2k|\log(\|A\|_{\textup{op}})|+2k|\log(\|B\|_{\textup{op}})|-2k\log(\lambda_{\textup{min}}(B))-2k\log(\lambda_{\textup{min}}(A)). 
}
\ee
\end{lemma}
\begin{proof}
Let $p\in[N]\cup\{0\}$ denote the number of positive eigenvalues of $A$, that is
\been{
\label{eq:def_p}
\l_1(A)\geq\ldots\geq\l_p(A)\geq0\qquad \mbox{and}\qquad 0\geq\l_{p+1}(A)\geq\ldots\geq\l_N(A)\,. 
}
Since the matrix $A-B$ has rank less or equal to $k$, then $\lambda_j(A-B)=0$ for all $j\in\{k+1,\cdots,N-k\}$. Hence the Weyl inequality for the eigenvalues yields
\bea
\l_j(A)\geq \l_{j+k}(B)\,,&\quad&1\leq j\leq N-k\,,\label{eq:weil1}\\
\l_{j}(A)\leq \l_{j-k}(B)\,,&\quad& k+1\leq j\leq N\,.\label{eq:weil2}
\eea
Then by \eqref{eq:def_p}
\bea
0\leq \l_j(A)\leq \l_{j-k}(B)\,,&\quad&\textup{if $p>k$ and } k+1\leq j\leq p\,,\label{eq:weil3}\\
0\geq \l_{j}(A)\geq \l_{j+k}(B)\,,&\quad&\textup{if $p<N-k$ and } p+1\leq j\leq N-k\,.\label{eq:weil4}
\eea
For the remaining eigenvalues, we consider the upper bound
\been{
\label{eq:UBZK}
|\l_j(A)|\leq  \|A\|_{\textup{op}},\quad \forall j\in \{1,\cdots,k\}\cup\{N-k+1,\cdots,N\}
} 
and the lower bound
\been{
\label{eq:UBZK2}
|\l_j(B)|\geq \lambda_{\textup{min}}(B),\qquad j\in[N].
} 
We distinguish few cases.
\begin{itemize}

\item[I.] $p=0$. We have
\been{\label{eq:daqua000}\aled{
NL(A)&\leq \sum_{j=1}^{N-k} \log \(\left|\l_{j+4}\left( B\right)\right|\)+\sum^N_{i=N-k+1}\log\left(|\l_{i}\left( A\right)|\right)\\
&\leq N L(B)-k\log\left(\l_{\textup{min}}\left( B\right)\right)+k\log\(\|A\|_{\textup{op}}\).
}
}

\item[II.] $p\in[k]$. We have
\been{\label{eq:daqua00}\aled{
NL(A)&\leq \sum_{j=p+1}^{N-k} \log \(\left|\l_{j+k}\left( B\right)\right|\)+\sum^p_{j=1}\log\left(|\l_{j}\left( A\right)|\right)+\sum^N_{j=N-k+1}\log\left(|\l_{j}\left( A\right)|\right)\\
&\leq N L(B)-(p+k)\log\left(\l_{\textup{min}}\left( B\right)\right)+(p+k)\log\(\|A\|_{\textup{op}}\)
}}
and $p+k\leq 2k$.

\item[III.] $p\in \{k+1,\cdots,N-k\}$. We have
\been{\label{eq:daqua0}\aled{
NL(A)&\leq \sum_{j=k+1}^{p} \log\(\left|\l_{j-k}\left( B\right)\right|\)+\sum_{j=p+1}^{N-k} \log \(\left|\l_{j+k}\left( B\right)\right|\)+\sum^k_{i=1}\log\left(|\l_{i}\left( A\right)|\right)+\sum^{k-1}_{i=0}\log\left(|\l_{N-i}\left( A\right)|\right)\\
&=N L(B)-\sum^{p+k}_{j=p-k+1}\log(|\l_{j}\left( B\right)|)+2k\log\(\|A\|_{\textup{op}}\)\\
&=N L(B)-2k\log(\l_{\textup{min}}\left( B\right))+2k\log\(\|A\|_{\textup{op}}\).
}}

\item[IV.] $p\in \{N-k+1,\cdots,N\}$. We have
\been{\label{eq:daqua0000}\aled{
N L(A)&\leq \sum_{j=k+1}^{p} \log \(\left|\l_{j-k}\left( B\right)\right|\)+\sum^{k}_{i=1}\log\left(|\l_{i}\left( A\right)|\right)+\sum^{N}_{j=p+1}\log\left(|\l_{j}\left( A\right)|\right)\\
&\leq N L(B)-(N-p+k)\log\left(\l_{\textup{min}}\left(B\right)\right)+(N-p+k)\log\(\|A\|_{\textup{op}}\)
}}
and $N-p+k\leq 2k$.
\end{itemize}
From the above cases, we deduce
\been{\label{eq:diffAB}
N(L(A)-L(B))\leq 2k\log(\|A\|_{\textup{op}})-2k\log(\l_{\min}(B)).
}
Reverting the role of $A$ and $B$ we obtain also
\been{\label{eq:diffBA}
N(L(B)-L(A))\leq2k\log(\|B\|_{\textup{op}})-2k\log(\l_{\min}(A)).
}
Combining \eqref{eq:diffAB} and \eqref{eq:diffBA} we get
\be
N|L(A)-L(B)|\leq2k\max\(\log(\|A\|_{\textup{op}})-\log(\l_{\min}(B)),\,\,\log(\|B\|_{\textup{op}})-\log(\l_{\min}(A))\),
\ee
whence (\ref{eq:erstatz-compl}) follows.
\end{proof}

Recall that $M^{([k],[k])}$ is the $(N-k)\times(N-k)$ matrix obtained by removing the first $k$ rows and the first $k$ columns from $M$, i.e.
\been{
\label{eq:def_Ap}
(M^{([k],[k])})_{ij}:=M_{i+k,j+k},\quad \forall (i,j)\in [N-k]^2.
}

\begin{lemma}
\label{lemma:erstatz2}
Let $k,N\in \N$ with $N>2 k$. Let $A$ be a $N\times N$ real symmetric matrix such that $A^{([k],[k])}$ is invertible. 
It holds
\be\label{eq:erstatz2}
N|L(A)-L(A^{([k],[k])})|
\leq 4k\log 2+8k|\log(\|A\|_{\textup{op}})|-4k\log(\l_{\min}(A))-4k\log(\l_{\min}(A^{([k],[k])})).
\ee
\end{lemma}
\begin{proof}
Let
\been{
\widetilde{A}:=P_{[k]}+P^\perp_{[k]}AP^\perp_{[k]}
=\left(
\begin{array}{c|c}
I_k & 0\\[4pt]
\hline\\[-10pt]
0 & A^{([k],[k])}\\[4pt]
\end{array}
\right).
}
Thus
\been{
\det(\widetilde{A})=\det(A^{([k],[k])})
}
and so
\been{
L(\widetilde{A})=\frac{1}{N}\log\(|\det(\widetilde{A})|\)=\frac{1}{N}\log\(|\det(A^{([k],[k])})|\)=L(A^{([k],[k])}).
}
We  have
\be\label{eq:A-tildeA}
\widetilde{A}-A=P_{[k]}+P_{[k]}AP_{[k]}-P_{[k]}A-AP_{[k]}=P_{[k]}(I_N+A)P_{[k]}-(P_{[k]}A+AP_{[k]}). 
\ee
The first summand on the r.h.s. of (\ref{eq:A-tildeA}) is a upper diagonal square block of rank $k$, while the second one has a lower diagonal $N-k\times N-k$ block with all zero entries, so it has rank at most $2k$. 
Then also their sum will have the same bound for the rank, that is $\textup{rank}(\widetilde{A}-A)\leq 2k$. By Lemma \ref{lemma:erstatz} we get
\be\label{eq:leftright}
N|L(\widetilde{A})-L(A)|\leq 4k |\log(\|A\|_{\textup{op}})|+4k |\log(\|\widetilde{A}\|_{\textup{op}})|-4k \log(\l_{\min}(A))-4k \log(\l_{\min}(\widetilde{A})). 
\ee
We have
\been{
\label{eq:lmaAA}
\aled{
\|\widetilde{A}\|_{\textup{op}}\leq \|P_{[k]}\|_{\textup{op}}+\|P^{\perp}_{[k]}AP^{\perp}_{[k]}\|_{\textup{op}}\leq 1+\|A\|_{\textup{op}}
}
}
since $\|P_{[k]}\|_{\textup{op}}\leq 1$ and $\|P^{\perp}_{[k]}AP^{\perp}_{[k]}\|_{\textup{op}}\leq \|A\|_{\textup{op}}$. Moreover 
\been{
\l_{\min}(\widetilde{A})=\min\(1,\l_{\min}(A^{([k],[k])})\). 
}
Thus
\been{
\label{eq:lmii}
|\log(\l_{\min}(\widetilde{A}))|\leq |\log(\l_{\textup{min}}(A^{([k],[k])}))|
}
Hence, combining \eqref{eq:leftright}, \eqref{eq:lmaAA} and \eqref{eq:lmii} and using the inequality $\log(1+|x|)\leq \log 2+|\log(x)|$, we end the proof. 
\end{proof}

Recall that
\be\label{eq:defZsect4}
Z=\frac{\b}{\sqrt N} J-D=Z_*-A.
\ee

\begin{lemma}\label{eq:tail_cond1}
Let $Z$ be defined as in \eqref{eq:defZsect4}. 
For $N$ large enough, it holds
\been{
\label{eq:log_tail1}
\mathbb{E}[|\log (\|Z\|_{op})|] \lesssim \log N+|\log(\|D\|_{op}+\sqrt{2}\b)|.
}
\end{lemma}
\begin{proof}
Let $x_0:=\|D\|_{op}+\sqrt{2}\b$. We split
\been{
\label{eq:log_split}
\aled{
\mathbb{E}[|\log (\|Z\|_{\textup{op}})|]
&=\mathbb{E}\left[-\log (\|Z\|_{\textup{op}}) 1_{\{\|Z\|_{\textup{op}}<N^{-2}]\}}\right]\\
&+\mathbb{E}\left[|\log (\|Z\|_{\textup{op}}) |1_{\{\|Z\|_{\textup{op}}\in (N^{-2},x_0]\}}\right]\\
&+\mathbb{E}\left[\log (\|Z\|_{\textup{op}}) 1_{\{\|Z\|_{\textup{op}} > x_0\}}\right].
}
}
We have
\been{
 \label{eq:1comb0}
\aled{
&\mathbb{E}\left[-\log (\|Z\|_{\textup{op}}) 1_{\{\|Z\|_{\textup{op}}<N^{-2}]\}}\right]=\mathbb{E}\left[\log (\|Z\|^{-1}_{\textup{op}}) 1_{\|Z\|^{-1}_{\textup{op}}>N^2]}\right]\\
&=2\log(N)+\int^{\infty}_{N^2}\frac{dt}{t}\P(\|Z\|^{-1}_{\textup{op}}>t)\leq 2\log(N)+\int^{\infty}_{N^2}\frac{dt}{t}\P(\l^{-1}_{\min}(Z)>t)\\
&\leq 2\log(N)+N^2 C \int^{\infty}_{N^2}\frac{dt}{t^2}\lesssim \log N\,,
}
}
where we used Lemma \ref{lem:lminP} in the last step.
Moreover
\been{
\label{eq:1comb1}
\mathbb{E}\left[|\log (\|Z\|_{\textup{op}}) |1_{\{\|Z\|_{\textup{op}}\in (N^{-2},x_0]\}}\right]\leq \max\{2\log(N),|\log (x_0)|\}\leq  2\log N+|\log(\|D\|_{op}+\sqrt{2}\b)|.
}
Finally
\been{
\label{eq:1comb2}
\aled{
&\mathbb{E}\left[-\log (\|Z\|_{\textup{op}}) 1_{\{\|Z\|_{\textup{op}}>x_0]\}}\right]=\int^{\infty}_{\|D\|_{op}}\frac{dt}{t+\sqrt 2\b}\P(\|Z\|_{\textup{op}}>\sqrt2\b+t)\\
&\lesssim \frac{1}{x_0}\int^{\infty}_{0}dte^{-N\frac{t^2}{2\b^2}}\lesssim \frac{1}{\sqrt{N}}.
}
}
Combining the bounds \eqref{eq:1comb0}, \eqref{eq:1comb1}, and \eqref{eq:1comb2} we end the proof.
\end{proof}

The last step before the proof of Proposition \ref{lem:concentration} is the following.

\begin{lemma}\label{lemma:laststep}
Let $k,N\in\N$ with $k<N$. Let $J_{N-k}$ be a $(N-k)\times (N-k)$ GOE matrix, $O$ be a deterministic $N\times N$ orthogonal matrix and let $A\in \R^{(N-k)\times (N-k)}$ be a deterministic symmetric matrix of rank $p$ with $0\leq 2p\leq N-k$ and
\been{
\label{eq:def_Zstar_mat}
Z_*=\frac{\b}{\sqrt{N}}J_{N-k}-(ODO^T)^{([k],[k])}+A.
}
It holds
\be
\left|\E[L(Z_*+i\e )]-\E[L(Z+i\e )]\right|\lesssim \frac{p+k}{N}\(|\log(\|D\|_{op}+\sqrt{2}\b)|+\log N+|\log\e|\)+\frac{p}{N}|\log\|A\|_{op}|.
\ee
\end{lemma}
\begin{proof}
We have
\been{
\label{eq:Zpar_equiv}
Z_{*}\eqd (OZO^T)^{([k],[k])}+A.
}
The eigenvalues of $Z+I_Ni\e$ and $OZO^T$ are the same and $L(Z)=L(OZO^T)$. By the triangular inequality
\bea
&&|\E[L(OZO^T+i\e)]-\E[L((OZO^T)^{([k],[k])}+A+i\e )]|\nn\\
&\leq& \left|\E[L(OZO^T+i\e )]-\E[L((OZO^T)^{([k],[k])}+i\e )]\right|\label{eq:triZst1}\\
&+&\left|\E[L((OZO^T)^{([k],[k])}+i\e )-L((OZO^T)^{([k],[k])}+A+i\e )]\right|\label{eq:triZst2}.
\eea

By Lemma \ref{lemma:erstatz2} and by the obvious inequalities $\l_{\min}(\cdot)\leq \|\cdot\|_{op}$ and $\|(\cdot)^{([k],[k])}\|_{op}\leq\|\cdot\|_{op} $ we have
\be\label{eq:ersatzeps1}
\eqref{eq:triZst1}\leq \frac1N\(8k\log2+16k\E[|\log\|Z+I_Ni\e\|_{op}|]\)\leq \frac1N\(8k\log2+16k\E[|\log(\|Z\|_{op}+\e)|]\).
\ee
Similarly, since $A$ has rank $p\leq (N-k)/2$, Lemma \ref{lemma:erstatz} gives
\be\label{eq:ersatzeps1}
\eqref{eq:triZst2}\leq \frac1N\(8p\E[|\log(\|Z\|_{op}+\e)|]+4p\log\|A\|_{op}\).
\ee
By the simple inequality 
\be\label{eq:uselog}
|\log(|a|+\e)|\leq |\log(|a|)|+|\log(\e)|+\log2
\ee 
and Lemma \ref{eq:tail_cond1} we get further
\bea
\!\!\!\!\!\!\!\!\eqref{eq:triZst1}\!\!\!\!&\lesssim&\!\!\!\!\frac kN\(1+\log N+|\log\e|+\log(\|D\|_{op}+\sqrt2\b)\)\label{eq:endP1}\\
\!\!\!\!\!\!\!\!\eqref{eq:triZst2}\!\!\!\!&\lesssim& \!\!\!\!\frac pN\(1+\log N + |\log\e|+\log(\|D\|_{op}+\sqrt2\b)+|\log\|A\|_{op}|\).\label{eq:endP2}
\eea
Combining (\ref{eq:triZst1}), (\ref{eq:triZst2}), (\ref{eq:endP1}), (\ref{eq:endP2}) ends the proof.
\end{proof}

\

\begin{proof}[Proof of Proposition \ref{lem:concentration}]
We split
\bea
&&\left|\frac{1}{Nr}\log\(\E\[e^{r NL(Z_*+i\e )}\]\)-\int_{\R}\log(|x|)\mu(dx)\right|\nn\\
&\leq&  \left|\frac{1}{Nr}\log\(\E\[e^{r NL(Z_*+i\e )}\]\)-\E[L(Z_*+i\e )]\right|\label{eq:Allconc1}\\
&+&\left|\E[L(Z_*+i\e )]-\E[L(Z+i\e )]\right|\label{eq:Allconc2}\\
&+&\left|\E[L(Z+i\e)]-\int_{\R}\log(|x|)\mu(dx)\right|\label{eq:Allconc3}.
\eea

The first summand can be estimated by the concentration inequality \cite[Theorem 2.3.5]{AGZ} repeating the same steps as in the proof of Proposition \ref{lem:concentration0} (see displays (\ref{eq:concentration30})-(\ref{eq:concentation40})). We have
\be
\eqref{eq:Allconc1}\leq \frac{r\b^2}{N^{\frac45}}. 
\ee
The second summand is estimated by Lemma \ref{lemma:laststep}, which gives for $\e\geq N^{1/10}$, $\|D\|_{op}\leq e^{N^{1-\d}}$ and $\|A\|_{op}\leq e^{N^{1-\d}}$
\be
\eqref{eq:Allconc2}\lesssim (p+k)N^{-\d}.
\ee

The third summand is estimated by Proposition \ref{lem:concentration0} which for $\e\in(N^{-\frac{1}{10}},N^{-\frac{1}{12}})$ gives
\been{
\label{eq:IIIconc}
\eqref{eq:Allconc3}\lesssim \frac{1}{N^{\min\(\frac{1}{24},\d\)}}. 
}
Gathering these contributions we prove the statement.
\end{proof}


\section{Preliminaries on the Bray-Moore formula}\label{sect:prel}

In this section we begin the proof of Theorem \ref{thm:MainDet}.
First of all, let us recall some definitions from Section \ref{sect:intro}. The Hessian of the TAP free energy is given by
\been{
\label{eq:nabla2H-expl2}
\nabla^2 F_{\TAP}(m)=\frac{\b J}{\sqrt N}-D_N(m)+4\frac{\b^2}{N}mm^T\,,
}
\generalxi{
$$
    F_{\TAP}(m):=\b H_{N}\left(m\right)+h\sum_{i=1}^N m_i+\ent(m)+N \On(Q(m))\,
$$
where
$$
    \On(q):=\frac{\beta^2}{2}\left(\xi(1)-\xi'(q)(1-q)-\xi(q)\right).
$$
}where
\be\label{eq:defD2}
    D_N(m):=\diag(\{(1-m_i^2)^{-1}+2\b^2(1-Q(m))\}_{i\in[N]})\,. 
\ee
\generalxi{
    \be\label{eq:defD}
        D_N(m)
        :=
        \diag(
            \{(1-m_i^2)^{-1}
            +
            \b^2(1-Q(m))\xi''(Q(m))\}_{i\in[N]}
        )\,. 
    \ee
}
From now on we denote
\be\label{eq:def_Z_mat}
Z:=\frac{\b J}{\sqrt N}-D_N(m)\,.
\ee

Consider the event 
$\Omega(m):=\{J\,:\,\nabla F_{\TAP}(m)=0\}$. We analyse the law of the Hessian conditionally on $\Omega(m)$, which we denote by
\be\label{eq:def-cond-Hess}
\nabla^2 F_{\TAP}(m)\big|_{\Omega(m)}\,.
\ee

The determinant of the matrix \eqref{eq:def-cond-Hess} is what we call the Bray-Moore determinant.

In what follows we consider $m\neq0$. Moreover we need to impose a cutoff which keeps $m$ away from the boundary of $[-1,1]^N$, where the TAP free energy is singular. For $\a>0$ we set
\be\label{eq:def_EdaN}
L_{\a,N}=[\,-1+e^{-\alpha \sqrt N},1-e^{-\alpha \sqrt N}\,]^N\setminus \{0\}
\ee
and we convey that $m\in L_{\a,N}$ from now on. 

Since the matrix $D_N(m)$ depends on $m\in L_{\a,N}$ the additive convolution of the empirical law of the eigenvalues of $D_N(m)$ with the semicircle law and its Stieltjes transform will also depend on $m$. They are given by
\been{
\label{eq:def_mu_m}
    \mu_m := \nu_{D_N} \boxplus \mu_{\rm{sc}, \sqrt{2}\beta}
}
\generalxi{
    \been{
    \label{eq:def_mu_m}
        \mu_m := \nu_{D_N} \boxplus \mu_{{\rm{sc}}, \sqrt{\xi''(Q(m))}\beta}
    }
}
and
\be\label{eq:stile2}
\widehat \mu_m(z)=\frac1N\sum_{i\in[N]}\frac{1}{(1-m_i^2)^{-1}+\b^2(1-Q(m))-\sqrt 2z-2\b^2\widehat \mu(z)}\,. 
\ee
We also remark that, even though this is not made explicit in our notation, the measure $\mu_m$ depends on $N$ through the choice of the point $m\in(-1,1)^N$.

The connection between the additive convolution and the Bray-Moore formula (see \eqref{eq:laformula}, \eqref{eq:defUps}, \eqref{eq:MainDet}) is provided by the following statement.
\begin{proposition}\label{lemma:vardet}
Let $\a>0$, and $N\in\N$ large enough. For any $m\in L_{\a,N}$ it holds
\be\label{eq:vardet}
    \int\log|x|\mu_{m}(dx)
    =
    \min_{g\leq 1-Q(m)} \left(
        \b^2g^2
        +
        \frac1N\sum^N_{i=1}\log(
            (1-m_i^2)^{-1}
            +
            2\b^2(1-Q(m))
            -
            2\b^2g)
    \right)\,
\ee
\generalxi{
    \be\label{eq:vardet}
        \int\log|x|\mu_{m}(dx)
        =
    \min_{g\leq 1-Q(m)} \left(
        \frac{\xi''(Q(m))}{2}\b^2g^2
        +
        \frac1N\sum^N_{i=1}\log(
            (1-m_i^2)^{-1}
            +
            \b^2(1-Q(m))\xi''(Q(m))
            -
            \xi''(Q(m))\b^2g)
    \right)\,
    \ee
}
and the minimum is attained at a unique point in $[0, 1-Q(m)]$. 
Moreover the $\min$ in \eqref{eq:vardet} is attained at $g=1-Q(m)$ if and only if the Plefka condition 
\be\label{eq:PlefkaLim}
2\b^2\sum^N_{i=1} (1-m_i^2)^2\leq1
\ee
holds.
\end{proposition}
\begin{proof}
Let us set for brevity
\bea
\label{eq:R_det}
    R_N(g)&:=&
        \b^2g^2
        +
        \frac1N\sum^N_{i=1} \log(
            (1-m_i^2)^{-1} + 2\b^2(1-Q(m)) - 2\b^2g
        )\,,\label{eq:defRK}\\
    F_N(g)&:=&
        \frac{1}{N}\sum_{i=1}^N \frac{1}{
            (1-m_i^2)^{-1} + 2\b^2(1-Q(m)) - 2\b^2g
        } \,.\label{eq:defFK}
\eea
The function $R_N$ is piece-wise differentiable and it diverges for $g\to\pm\infty$. Therefore it attains its minima in the interior of its support. 
Moreover $F_N$ is a rational function with $K$ poles $p_{1},\ldots, p_N$ all lying on the positive semi-axis, which we order increasingly. 
We have
\be\label{eq:criticRK}
    R_N'(g)=2\b^2\left(g-F_N(g)\right)\,.
\ee

It is easy to see that there is no solutions of $R_N'(g)=0$ on the negative semi-axis. Moreover in \cite[Lemma 4.8, Lemma 4.9]{Mon1} it is proven that
\begin{itemize}
\item[1.] In $(0,p_1)$ there are at most two solutions of the equation $R_N'(g)=0$, namely $g=1-Q(m)$ and $g=g_0$ with $g_0\leq 1-Q(m)$;
\item[2.] It holds that $R_N(g_0)\leq R_N(1-Q(m))$;
\item[3.] It holds that $\int\log|x|\mu_{m}(dx)=R_N(g_0)$.
\end{itemize}
We conclude
$$
\int\log|x|\mu_{m}(dx)=\min_{g\leq 1-Q(m)} R_N(g)\,. 
$$

Now assume that (\ref{eq:PlefkaLim}) holds. Computing the first derivative of the function $F_N$ we get
\be
F'_N(g)=2\b^2\sum_{i=1}^N\frac{1}{((1-m_i^2)^{-1}+2\b^2(1-Q(m))-2\b^2g)^2}\,
\ee
and clearly for any $g< 1-Q(m)$
$$
F'_N(g)\leq 2\b^2\sum_{i=1}^N(1-m_i^2)^2\,. 
$$
Since $g_0$ solves $x=F_N(x)$, it must be $F'_N(g)\big|_{g=g_0}=1$, hence if $g_0< 1-Q(m)$ the condition (\ref{eq:PlefkaLim}) is violated. Thus it must hold that $g_0=1-Q(m)$. 
\end{proof}

Therefore Theorem \ref{thm:MainDet} will follow from the subsequent result, extending Theorem \ref{thm:GOE_plus_D_det}. The rest of the manuscript will be devoted to its proof. 

\begin{theorem}\label{thm:MainDet2}
Let $\a>0$, $N\in\N$ large enough, $m\in L_{\a,N}$. Consider
\be\label{eq:defu-v-sect4}
v:=\atanh m-h+2\b^2m(1+Q(m))-D_N(m)m\,.
\ee
Let  $\mu_m$ be the measure defined in \eqref{eq:def_mu_m}. Then there exists a constant $c:=c(\a)$ such that
\be\label{eq:MainDet2}
\left|\frac1N\log \mathbb{E}\left[|\det\nabla^2 F_{\TAP}(m)|\,\big|\,\Omega(m)\right]-\frac1N\log\(\frac{\|v\|^2_2}{\|m\|^2_2}+\frac{|(m,v)|}{\|m\|^2_2}-\frac{|(m,v)|^2}{\|m\|^4_2}\)- \int \log|x|\mu_m(dx)\right|\leq\frac{c}{N^{\frac{1}{120}}}\,,
\ee
where the right hand side is equal to $-\infty$ if $v=0$.
\end{theorem}

We give one technical lemma which allows us to use the results of the previous sections for $m\in L_{\a,N}$. 

\begin{lemma}
\label{lem:prop2}
Let $m\in L_{\a,N}$. There exists $N_0:=N_0(\a,\b,|h|)$ such that for any $N> N_0$ it holds
\been{
\label{eq:v_ineq}
 \|v\|_{2}\leq e^{3\a\sqrt{N}}.
 }
Furthermore
\been{
\label{eq:D_UB}
\|D_N(m)\|_{\textup{op}}\leq 2e^{2\a\sqrt{N}}.
 }
\end{lemma}
\begin{proof}
Combining \eqref{eq:defD2} and \eqref{eq:defu-v-sect4} we see that
\be
v_i=\atanh m_i-h-\frac{m_i}{1-m_i^2}+4\b^2Q(m)m_i,\qquad i\in[N]. 
\ee

If $m\in L_{\a,N}$ then for any $i\in[N]$ it holds
\been{
\frac{1}{1-m^2_i}\leq e^{2\a\sqrt{N}}
}
and
\been{
|2\atanh(m_i)-2h|\leq |\log 2-\log(1-m^2_i)|+2|h|\leq 2\alpha\sqrt{N}+\log 2+2|h|\leq e^{2\a\sqrt{N}}
}
for $N> (|h|+3+4\b^2)/ \a^2$. Moreover $4\b^2Q(m)m_i\leq 4\b^2$. Thus
\been{
\|v\|^2_2\leq N\max_{i\in[N]}\(\atanh m_i-h-\frac{m_i}{1-m_i^2}+4\b^2Q(m)m_i\)^2\leq 9Ne^{4\alpha \sqrt{N}}\leq e^{6\alpha \sqrt{N}}.
}
Moreover
\been{
\|D_N(m)\|_{\textup{op}}\leq  \max\left\{\left|\frac{1}{1-m^2_i}+2\b^2(1-Q(m))\right|,i\in[N]\right\}\leq 2e^{2\alpha \sqrt{N}}.
}
\end{proof}


\section{Representations of the Bray-Moore determinant}\label{sect:noN}

In this section we complete the proof of Theorem \ref{thm:MainDet2}.

First we establish some convenient representations of the Bray-Moore determinant. Recall
\be\label{eq:defu-v-sect5}
    u = \atanh m - h + 2\b^2m(1-Q(m))\,,
    \qquad
    v = u + 4\b^2Q(m) m - D_N(m)m\,.
\ee
\generalxi{
    \be\label{eq:defu-v-sect5}
        u = \atanh m - h + \b^2m(1-Q(m))\xi''(Q(m))\,,
        \qquad
        v = u 
        +
        4\b^2Q(m) m
        -
        D_N(m)m\,.
    \ee
}

The following lemma starts our considerations. Observe that differently to \cite[formula (4.2)]{Mon1} we decompose the Hessian as a $(N-1)$-rank random matrix plus a deterministic correction of small rank. 

\begin{lemma}\label{lem:cond_law}
For any $m\in (-1,1)^N\setminus\{0\}$ we have
\be\label{eq:decomp1}
\nabla^2 F_{\TAP}(m)\big|_{\Omega(m)}\overset{d}{=}P^\perp_{m}\,Z\,P^\perp_{m}+K(m)
\ee
where $K(m)$ is given by
\been{ \label{eq:K}
K(m):=\frac{1}{\|m\|_2^2}m v^T+\frac{1}{\|m\|_2^2}v m^T-\frac{(m,v)}{\|m\|_2^4}m m^T\,,
}
and the vector $v$ is defined in \eqref{eq:defu-v-sect5}.
\end{lemma}
\begin{proof}
We start by
\be
\nabla F_{\TAP}(m)=\frac{\b}{\sqrt N}Jm-\atanh(m)+h-2\b^2m(1-Q(m))=\frac{\b}{\sqrt N}Jm-u
\ee
and
\been{
    \label{eq:nabla2H-expl2-q}
    \nabla^2 F_{\TAP}(m)
    =
    \frac{\b J}{\sqrt N}
    - D_N(m)
    + 4\b^2 Q(m)P_m\,.
}

Thus conditionally on $\Omega(m)$ it holds that $u=\frac{\b}{\sqrt N}Jm$, hence
\be
\frac{\b}{\sqrt N} J\big|_{\Omega(m)}\eqd P_m^\perp\(\frac{\b}{\sqrt N} J\)P_m^\perp+\frac{1}{\|m\|_2^2}\(mu^T+um^T\)-\frac{(m,u)}{\|m\|_2^2}P_m
\ee
(see for instance \cite[Lemma 4.1]{Mon1}). Moreover
\be
D_N(m)= P_m^\perp D_N(m)P_m^\perp- \frac{ (m,D_N(m)m)}{\|m\|^2_2}P_m+D_N(m)P_m+P_m D_N(m),
\ee
whence
\been{
\label{eq:nabla2H-expl3}
\aled{
\nabla^2 F_{\TAP}(m)\big|_{\Omega(m)}&\eqd P_m^\perp ZP_m^\perp +\frac{ (m,D_N(m)m)}{\|m\|^2_2}P_m- D_N(m)P_m-P_m D_N(m)\\
&+\frac{1}{\|m\|_2^2}\(mu^T+um^T\)+\(-\frac{(m,u)}{\|m\|_2^2}+4\b^2Q(m)\)P_m.
}
}
By the definition of $v$ (\ref{eq:defu-v-sect5}) we have
\bea
\(-\frac{(m,u)}{\|m\|_2^2}+4\b^2Q(m)\)P_m&=&-\frac{(m,v)}{\|m\|_2^2}P_m-\frac{(m,D_N(m)m)}{\|m\|_2^2}P_m+8\b^2Q(m)P_m\label{eq:pozzi1}\\
\frac{1}{\|m\|_2^2}\(mu^T+um^T\)&=&\frac{1}{\|m\|_2^2}\(mv^T+vm^T\)-8\b^2Q(m)P_m+D_N(m)P_m+P_m D_N(m)\label{eq:pozzi2}
\eea
When we plug \eqref{eq:pozzi1}, \eqref{eq:pozzi2} into \eqref{eq:nabla2H-expl3} we get \eqref{eq:decomp1}, \eqref{eq:K}. 
\end{proof}

The above lemma implies that for any $m\in(-1,1)^N\setminus\{0\}$
\be\label{eq:centre}
\E[|\det\(\nabla^2 F_{\TAP}(m)\)|\,|\,\Omega(m)]=\E[|\det\(P^{\perp}_{m}ZP^{\perp}_{m}+K(m)\)|].
\ee

The last lemma accounts for the the last sentence of Theorem \ref{thm:MainDet2}. 
\begin{lemma}\label{lemma:v=0}
Let $m\in L_{\a,N}$ be such that $v=0$. Then
\been{
\label{eq:zero_dett}
\det\(P^{\perp}_{m}ZP^{\perp}_{m}+K(m)\)=0.
}
\end{lemma}
\begin{proof}
If $v=0$, then $K(m)=0$. Thus
\been{
\(P^{\perp}_{m}ZP^{\perp}_{m}+K(m)\)m=P^{\perp}_{m}Z(P^{\perp}_{m}m)=0.
}
It follows that the matrix $P^{\perp}_{m}ZP^{\perp}_{m}+K(m)$ must have at least one vanishing eigenvalue, whence (\ref{eq:zero_dett}).
\end{proof}

Due to Lemma \ref{lemma:v=0} we will assume hereafter that $v\neq0$.

\

Before going on, few more words on the notation.
Recall that $M^{([p],[p])}$ is the $(N-p)\times(N-p)$ matrix obtained by removing the first $p$ rows and the first $p$ columns from the matrix $M$, i.e.
\been{
\label{eq:def_Ap}
(M^{([p],[p])})_{ij}:=M_{i+p,j+p},\quad \forall (i,j)\in [N-p]^2.
}
We shorten $M^{(1,1)}=M^{([1],[1])}$.
Recall also that $e^i_N\in \R^N$ is the $i$-th standard basis vector and 
\been{
\label{eq:def_EM}
P_{[\ell]}=\sum^\ell_{j=1}P_{e^j_N}.
}

Let us now define the vector
\been{
\label{eq:x}
x=v-\frac{(m,v)}{\|m\|_2^2}m.
}
Note that, assuming $v\neq 0$ not parallel to $m$, it is $x\neq 0$ and $(x,m)=0$. Moreover an easy computation gives
\been{\label{eq:easycompK}
K(m)=\frac{1}{\|m\|_2^2}m x^T+\frac{1}{\|m\|_2^2}x m^T+\frac{(m,v)}{\|m\|_2^4}m m^T\
}
We introduce the orthogonal projections
\been{
P_{mx}=\frac{mm^T}{\|m\|^2_2}+\frac{xx^T}{\|x\|^2_2}\,,\qquad P^{\perp}_{mx}=I_N-P_{mx}.
}
Let further $O_{m}$ be a $N\times N$ orthogonal matrix such that
\been{
\label{eq:def_Om}
O_{m}m:=(\|m\|,0,\cdots,0)^T=\|m\|e^1_N.
}
Moreover since $(x,m)=0$ there exists an orthogonal matrix $O_{mx}$ such that the first column is parallel to $m$ and the second is parallel to $x$, that is 
\been{
\label{eq:A_POP}
O_{mx}m=\|m\|_2e^1_N,\quad O_{mx}x=\|x\|_2e^2_N.
}
The crucial point is that the matrix $O_{mx}$ is deterministic, since $m$ and $x$ do not depend on $J$.

Let $J', J''$ be respectively a $N-1\times N-1$ and a $N-2\times N-2$ GOE matrix. We set
\bea
Z^{(\parallel)}&:=&\frac{\b}{\sqrt N}J'-(O_{m}D_N(m)O^T_{m})^{(1,1)}\label{eq:defZpara},\\
Z^{(\perp)}&:=&\frac{\b}{\sqrt{N}}J''-(O_{mx}D_N(m)O^T_{mx})^{([2],[2])}\label{eq:til_Z2_def}.
\eea

The next statement follows directly by applying the results of the previous sections.

\begin{proposition}\label{prop:logdetZfin}
Let $\a, \b>0$, $N\in\N$ and $m\in L_{\a,N}$. It holds
\bea
\left|\frac1N \log \E|\det Z^{(\parallel)}|-\int \log|x|\mu_m(dx)\right|&\lesssim&\frac{1}{(N-1)^{\frac{1}{60}}}\label{eq:logdetZfin1}\\
\left|\frac1N \log \E|\det Z^{(\perp)}|-\int \log|x|\mu_m(dx)\right|&\lesssim&\frac{1}{(N-2)^{\frac{1}{60}}}.\label{eq:logdetZfin2}
\eea
Moreover if $|(m,v)|\geq e^{-N^{1-\d'}}$ for some $\d'\in(0,1)$ then
\be\label{eq:logdetZfin3}
\left|\frac1N \log \E\left|\det \(Z^{(\parallel)}-\frac{\|x\|^2_2}{(m,v)}P_{e^1_{N-1}}\)\right|-\int \log|x|\mu_m(dx)\right|\lesssim \frac{1}{(N-1)^{\min\(\d,\frac{1}{60}\)}}.
\ee 
\end{proposition}
\begin{proof}
The bound \eqref{eq:logdetZfin1} follows from Corollary \ref{lem:almost_fin_det} with $k=1$, $p=0$, $A=0$, $\d=1/2$ (by Lemma \ref{lem:prop2}). 
The bound \eqref{eq:logdetZfin2} follows from Corollary \ref{lem:almost_fin_det} with $k=2$, $p=0$, $A=0$, $\d=1/2$ (by Lemma \ref{lem:prop2}). 
The bound \eqref{eq:logdetZfin3} follows from Corollary \ref{lem:almost_fin_det} with $k=1$, $p=1$,  
\be
A=\frac{\|x\|^2_2}{(m,v)}P_{e^1_{N-1}}
\ee
and $\d=\max(1/2,1-\d')$, since by Lemma \ref{lem:prop2} and definition \eqref{eq:x}
\be
\|A\|_{op}=\frac{\|x\|^2_2}{|(m,v)|}\leq \frac{\|v\|^2_2}{|(m,v)|}\leq e^{N^{1-\d'}+3\a\sqrt{N}}.
\ee
\end{proof}

Here is an obvious upshot of the above statement. 

\begin{lemma}\label{lemma:verysleepy}
Let $\a>0$, $N\in\N$ sufficiently large and $m\in L_{\a,N}$. It is 
\be\label{eq:verysleepy}
\frac{\E[|\det Z^{(\parallel)}|]}{\E[|\det Z^{(\perp)}|]}\lesssim e^{N^\frac{59}{60}}. 
\ee
\end{lemma}

We first examine the case in which the vector $v$ defined in (\ref{eq:defu-v-sect5}) is parallel to $m$. We set
\been{
\label{eq:def_Tp}
T^{\parallel}_{\a,N}:=\left\{m\in L_{\a,N}\setminus\{0\}:\,\|v\|_2=\frac{|(m,v)|}{\|m\|_2},\,v\neq0\right\}.
}
\begin{lemma}\label{prop:repr-det}
Consider $\a>0$, $v$ as in (\ref{eq:defu-v-sect5}) and $m\in T^{\parallel}_{\a,N}$. We have
\been{\label{eq:easy_det1}
\det\(P^{\perp}_{m}ZP^{\perp}_{m}+K(m)\)\eqd\frac{\|v\|_2}{\|m\|_2}\det\(Z^{(\parallel)}\).
}
\end{lemma}
\begin{proof}
Since $v$ is parallel to $m$, then (recall (\ref{eq:K}))
\been{
\nabla^2F_{\TAP}(m)\big|_{\Omega(m)}\eqd P^{\perp}_{m}ZP^{\perp}_{m}+K(m)=P^{\perp}_{m}ZP^{\perp}_{m}+\frac{(m,v)}{\|m\|^4_2}mm^T.
}
Let us shorten 
\been{\label{eq:defAB}
A:=P^{\perp}_m+\frac{(m,v)}{\|m\|^2_2}P_m,\qquad B:=P_m+P^{\perp}_{m}ZP^{\perp}_{m}.
}
Then $Z=A B$ implies
\been{
\label{eq:det_Zz}
\det(Z)=\det(A)\det(B)=\det(A)\det(O_mBO^T_m)
}
Note that
\been{
\label{eq:det_Az}
\det(A)=\frac{(m,v)}{\|m\|^2_2}=\frac{\|v\|_2}{\|m\|_2},
}
since $m \in T^{\parallel}_{\a,N}$. From the definition of $O_m$ \eqref{eq:def_Om}, we have
\been{
O_mP_mO^T_m=P_{e^1_N},\quad O_mP^{\perp}_mO^T_m=P^\perp_{e^1_N}.
}
As a consequence
\been{
O_mBO^T_m=P_{e^1_N}+(O_mP^{\perp}_mO^T_m)(O_mZO^T_m)(O_mP^{\perp}_mO^T_m)=P_{e^1_N}+P^\perp_{e^1_N}(O_mZO^T_m)P^\perp_{e^1_N}.
}
Thus $O_mBO^T_m$ is a Block diagonal matrix of this form
\been{
O_mBO^T_m=
\left(
\begin{array}{c|c}
 \begin{matrix} 1 \end{matrix} & \begin{matrix} 0&\cdots&0 \end{matrix} \\ \hline \begin{matrix} 0\\\vdots \\0 \end{matrix} & \begin{matrix} \,(O_mZO^T_m)^{(1,1)}\, \end{matrix}
\end{array}
\right).
}
Hence
\been{
\label{eq:det_Bz}
\det(O_mBO^T_m)=\det((O_mZO^T_m)^{(1,1)})\eqd\det(Z^{(\parallel)}).
}
So, combining \eqref{eq:det_Zz}, \eqref{eq:det_Az} and \eqref{eq:det_Bz} we end the proof.
\end{proof}

If $m$ and $v$ are not aligned the situation is more complicated. We need the following intermediate result.

\begin{lemma}\label{lem:det_formula0}
Let $\a>0$ and $m\in L_{\a,N}\setminus T^{\parallel}_{\a,N}$. 
Then
\been{
\label{eq:detY}
\det\(P^{\perp}_{m}ZP^{\perp}_{m}+K(m)\)\eqd 
\frac{(m,v)}{\|m\|_2^2}\det Z^{(\parallel)}-\frac{\|x\|_2^2}{\|m\|_2^2}\det Z^{(\perp)}.
}
\end{lemma}
\begin{remark}\label{rmk:clarity}
Let $Y$ be the $N\times N$ matrix with entries
\been{
\label{eq:def_Y0}
Y_{ij}:=\frac{(m,v)}{\|m\|^2_2}\delta_{1i}\delta_{1j}+\frac{\|x\|_2}{\|m\|_2}(\delta_{1i}\delta_{2j}+\delta_{2i}\delta_{1j})+(1-\d_{1i})(1-\d_{1j})(Z^{(\parallel)})_{i-1,j-1}.
}
This matrix will play a central role in the next two proofs. For sake of clarity we write it explicitly:
\been{
\label{eq:Y_rep}
Y=
\left(
\begin{array}{c|cccc}
\frac{(m,v)}{\|m\|^2_2} & \frac{\|x\|_2}{\|m\|_2}  & 0 & \cdots & 0 \\[4pt]
\hline\\[-10pt]
\frac{\|x\|_2}{\|m\|_2} & (Z^{(\parallel)})_{11}  & (Z^{(\parallel)})_{12} & \cdots & (Z^{(\parallel)})_{1,N-1}\\[4pt]
 0 & (Z^{(\parallel)})_{21}  & (Z^{(\parallel)})_{22} & \cdots & (Z^{(\parallel)})_{2,N-1}\\[4pt]
 \vdots & \vdots  &  \vdots  & \ddots & \vdots\\[4pt]
 0 & (Z^{(\parallel)})_{N-1,1}  & (Z^{(\parallel)})_{N-1,2} &  \cdots & (Z^{(\parallel)})_{N-1,N-1}  \\[4pt]
\end{array}
\right).
}
\end{remark}
\begin{proof}
We have
\been{
\label{eq:A_POP2}
(O_{mx}P_m\,O^T_{mx})=P_{e^1_N},\quad (O_{mx}P_{mx}\,O^T_{mx})=P_{[2]}
}
and
\been{
\label{eq:A_POP3}
(O_{mx}P^{\perp}_m\,O^T_{mx})=P^\perp_{e^1_N},\quad (O_{mx}P^{\perp}_{mx}\,O^T_{mx})=P^\perp_{[2]}. 
}

Using \eqref{eq:easycompK}, \eqref{eq:A_POP} and \eqref{eq:A_POP2} we get
\been{
\label{eq:def_Y1}
\aled{
O_{mx}K(m)O^T_{mx}&=\frac{1}{\|m\|_2^2}(O_{mx}m) (O_{mx}x)^T+\frac{1}{\|m\|_2^2}(O_{mx}x)(O_{mx}m)^T+\frac{(m,v)}{\|m\|_2^2}O_{mx}P_{m}O^T_{mx}\\
&=\frac{\|x\|_2}{\|m\|_2}\(e^1_N (e^2_N )^T+(e^2_N)(e^1_N)^T\)+\frac{(m,v)}{\|m\|_2^2}P_{e^1_N}.
}
}
Moreover from \eqref{eq:A_POP3} we get
\been{
\label{eq:def_Y2}
\aled{
O_{mx}P^{\perp}_{m}ZP^{\perp}_{m}O^T_{mx}&=(O_{mx}P^{\perp}_{m}O^T_{mx})(O_{mx}ZO^T_{mx})(O_{mx}P^{\perp}_{m}O^T_{mx})\\
&=P^\perp_{e^1_N}(O_{mx}ZO^T_{mx})P^\perp_{e^1_N}.
}
}
So, combining \eqref{eq:def_Y1} and  \eqref{eq:def_Y2}
\been{
\label{eq:def_Y}
\(O_{mx}\(P^{\perp}_{m}ZP^{\perp}_{m}+K(m)\)O^T_{mx}\)_{ij}=\frac{(m,v)}{\|m\|^2_2}\delta_{1i}\delta_{1j}+\frac{\|x\|_2}{\|m\|_2}(\delta_{1i}\delta_{2j}+\delta_{2i}\delta_{1j})+(1-\d_{1i})(1-\d_{1j})(O_{mx}ZO^{T}_{mx})_{ij}.
}
Clearly
\been{
\label{eq:comb2}
(O_{mx}ZO^T_{mx})\eqd \frac{\b}{\sqrt{N}}J-(O_{mx}D_N(m)O^T_{mx}),
}
hence
\been{
\label{eq:comb1}
(O_{mx}ZO^{T}_{mx})_{ij}=((O_{mx}ZO^{T}_{mx})^{(1,1)})_{i-1,j-1},\quad 2\leq i,j\leq N.
}
Then by \eqref{eq:comb2}
\be\label{eq:comb22}
(O_{mx}ZO^{T}_{mx})^{(1,1)}\eqd \frac{\b}{\sqrt{N}}J^{(1,1)}-(O_{mx}D_N(m)O^T_{mx})^{(1,1)}\\
\eqd \frac{\b}{\sqrt{N}}J'-(O_{mx}D_N(m)O^T_{mx})^{(1,1)}= Z^{(\parallel)}.
\ee
Combining \eqref{eq:def_Y0}, \eqref{eq:def_Y}, \eqref{eq:comb1} and \eqref{eq:comb22} we obtain 
\been{
\label{eq:detY-VERO}
\det\(P^{\perp}_{m}ZP^{\perp}_{m}+K(m)\)\eqd \det Y. 
}
Now we compute the determinant of $Y$ by applying the Laplace method twice. We have (see \eqref{eq:Y_rep})
\be
\det Y=\frac{(m,v)}{\|m\|_2^2}\det Z^{(\parallel)}-\frac{\|x\|_2^2}{\|m\|^2_2}\det [(Z^{(\parallel)})^{(1,1)}].
\ee
Since $(Z^{(\parallel)})^{(1,1)}\eqd Z^{(\perp)}$ (compare (\ref{eq:defZpara}) and \eqref{eq:til_Z2_def}) also the second relation in \eqref{eq:detY} is proved.
\end{proof}

\begin{remark}\label{rmk:XFrancesco}
Combining Lemma \ref{prop:repr-det} and Lemma \ref{lem:det_formula0} gives that the representation \eqref{eq:detY} holds for all $m\in L_{\a,N}$. If $v\neq0$ at least one of the two coefficients $(m,v)$ and $\|x\|_2$ must be different from zero. Since moreover $\det Z^{(\perp)}$ and $\det Z^{(\parallel)}$ vanish with zero probability, we have that (recall (\ref{eq:decomp1}))
\be\label{eq:sommaduelemmi}
P\(|\det \nabla^2 F_{\TAP}(m)\big|_{\Omega(m)}|>0\)=1.
\ee
Combining with Lemma \ref{lemma:v=0} we obtain
\been{
\label{eq:non_vanishing}
P\(|\det \nabla^2 F_{\TAP}(m)\big|_{\Omega(m)}|>0\)=
\begin{cases}
0\quad &\textup{if }v=0,\\
1\quad &\textup{otherwise}.
\end{cases}
}
\end{remark}

\begin{lemma}\label{lemma:neperpnepara}
Let $\a,\b>0$, $N\in\N$ and $m\in L_{\a,N}\setminus T^{(\parallel)}_{\a,N}$. If $m$ is such that $(v,m)\neq0$, then
\been{
\label{eq:det_formula0}
\det\(P^{\perp}_{m}ZP^{\perp}_{m}+K(m)\)\eqd \frac{(m,v)}{\|m\|^2_2}\det\(Z^{(\parallel)}-\frac{\|x\|_2^2}{(m,v)}P_{e^1_{N-1}}\). 
}
\end{lemma}
\begin{proof}
Let $Y$ be the matrix defined in Remark \ref{rmk:clarity}. 
If $(m,v)\neq 0$, then $Y_{11}\neq 0$. The Schur complement $Y/Y_{11}$ of the block $Y_{11}$ is
\been{
\label{eq:Schur_comp}
Y/Y_{11}=Y^{(1,1)}-\frac{Y^2_{12}}{Y_{11}}P_{e^1_{N-1}}=Z^{(\parallel)}-\frac{\|x\|^2_2}{(m,v)}P_{e^1_{N-1}}.
}

Then the Schur formula gives
\been{
\label{eq:detsh}
\det(Y)=Y_{11}\det\(Y/Y_{11}\)=\frac{(m,v)}{\|m\|^2_2}\det\(Z^{(\parallel)}-\frac{\|x\|^2_2}{(m,v)}P_{e^1_{N-1}}\).
}
Combining with \eqref{eq:detY-VERO} we prove \eqref{eq:det_formula0}.
\end{proof}

\begin{lemma}\label{lemma:immediate}
Let $\a,\b>0$, $N\in\N$ and $m\in L_{\a,N}$. If $m$ is such that $v$ is orthogonal to $m$, then
\been{
\label{eq:det_formula0}
\det\(P^{\perp}_{m}ZP^{\perp}_{m}+K(m)\)\eqd -\frac{\|x\|_2^2}{\|m\|^2_2}\det\(Z^{(\perp)}\). %
}
\end{lemma}
\begin{proof} 
Immediate plugging $(m,v)=0$ into \eqref{eq:detY}. 
\end{proof}

When $m$ is such that $v\perp m$ or $v\parallel m$ the proof of Theorem \ref{thm:MainDet2} is easier.

\begin{proposition}\label{prop:perppara}
If $m\in L_{\a,N}\setminus T_{\a,N}^{(\parallel)}$ is such that $v\perp m$ or $v\parallel m$ then
\be\label{eq:perppara}
\frac1N\log\E\left|\det\nabla^2 F_{\TAP}(m)\big|_{\Omega(m)}\right|=\frac1N\log\(\frac{\|v\|^2_2}{\|m\|^2_2}+\frac{|(m,v)|}{\|m\|_2^2}-\frac{(m,v)^2}{\|m\|_2^4}\)+\int \log|x|\mu_m(dx)+r_N\nn,
\ee
where $|r_N|\lesssim \frac{1}{N^{1/60}}$.

\end{proposition}
\begin{proof}
If $|(m,v)|=\|m\|_2\|v\|_2$ then
\be\label{eq:thenwhat?1}
\frac{\|v\|_2}{\|m\|_2}=\frac{\|v\|^2_2}{\|m\|^2_2}+\frac{|(m,v)|}{\|m\|_2^2}-\frac{(m,v)^2}{\|m\|_2^4}.
\ee

Therefore by Lemma \ref{lem:cond_law}, Lemma \ref{prop:repr-det}, (\ref{eq:logdetZfin1})  (and \eqref{eq:thenwhat?1}) if $m$ is such that $v\parallel m$ we have
\bea
\frac1N\log\E\left|\det\nabla^2 F_{\TAP}(m)\big|_{\Omega(m)}\right|&=&\frac1N\log\E\det|P^{\perp}_{m}ZP^{\perp}_{m}+K(m)|\nn\\
&=&\frac1N\log\frac{\|v\|_2}{\|m\|_2}+\frac1N\log\E|\det Z^{(\parallel)}|\nn\\
&=&\frac1N\log\frac{\|v\|_2}{\|m\|_2}+\int \log|x|\mu_m(dx)+r_N\nn\\
&=&\frac1N\log\(\frac{\|v\|^2_2}{\|m\|^2_2}+\frac{|(m,v)|}{\|m\|_2^2}-\frac{(m,v)^2}{\|m\|_2^4}\)+\int \log|x|\mu_m(dx)+r_N\nn,
\eea
where 
\be
|r_N|\lesssim \frac{1}{(N-1)^{1/60}}\,.
\ee

Similarly if $|(m,v)|=0$ then
\be\label{eq:thenwhat?2}
\frac{\|v\|^2_2}{\|m\|^2_2}=\frac{\|v\|^2_2}{\|m\|^2_2}+\frac{|(m,v)|}{\|m\|_2^2}-\frac{(m,v)^2}{\|m\|_2^4}.
\ee
Thus by Lemma \ref{lem:cond_law}, Lemma \ref{lemma:immediate}, (\ref{eq:logdetZfin2}) (and \eqref{eq:thenwhat?2}) if $m$ is such that $v\perp m$ we have
\bea
\frac1N\log\E\left|\det\nabla^2 F_{\TAP}(m)\big|_{\Omega(m)}\right|&=&\frac1N\log\E\det|P^{\perp}_{m}ZP^{\perp}_{m}+K(m)|\nn\\
&=&\frac1N\log\frac{\|v\|^2_2}{\|m\|^2_2}+\frac1N\log\E|\det Z^{(\perp)}|\nn\\
&=&\frac1N\log\frac{\|v\|^2_2}{\|m\|^2_2}+\int \log|x|\mu_m(dx)+r'_N\nn\\
&=&\frac1N\log\(\frac{\|v\|^2_2}{\|m\|^2_2}+\frac{|(m,v)|}{\|m\|_2^2}-\frac{(m,v)^2}{\|m\|_2^4}\)+\int \log|x|\mu_m(dx)+r'_N\nn,
\eea
where 
\be
|r'_N|\lesssim \frac{1}{(N-2)^{1/60}}\,.
\ee
\end{proof}

Finally we deal with the case in which $m$ and $v$ align generically. 

\begin{proposition}\label{prop:thelastone}
Let $\a>0$, $N\in\N$ and $m\in L_{\a,N}$ such that $0<|(m,v)|<\|v\|_2\|m\|_2$. Then
\be\label{eq:NOTperppara}
\frac1N\log\E\left|\det\nabla^2 F_{\TAP}(m)\big|_{\Omega(m)}\right|=\frac1N\log\(\frac{\|v\|^2_2}{\|m\|^2_2}+\frac{|(m,v)|}{\|m\|_2^2}-\frac{(m,v)^2}{\|m\|_2^4}\)+\int \log|x|\mu_m(dx)+r_N\nn,
\ee
where $|r_N|\lesssim \frac{1}{N^{1/120}}$.
\end{proposition}
\begin{proof}
Let us consider the cut-off
\be
T:=\{m\in L_{\a,N}\,:\, e^{-N^{119/120}}\|v\|_2\leq |(m,v)|<\|v\|_2\|m\|_2\}.
\ee
For $m\in T$ we have 
\be
\frac{|(m,v)|}{\|m\|^2_2}+\frac{\|v\|^2_2}{\|m\|_2^2}-\frac{(m,v)^2}{\|m\|_2^4}\leq \frac{|(m,v)|}{\|m\|^2_2}+\frac{\|v\|^2_2}{\|m\|_2^2}\leq \frac{|(m,v)|}{\|m\|_2^2}\(1+e^{N^{119/120}}\), 
\ee
hence
\be\label{eq:thenwhat?3}
0\leq \frac1N\log\(\frac{|(m,v)|}{\|m\|^2_2}+\frac{\|v\|^2_2}{\|m\|_2^2}-\frac{(m,v)^2}{\|m\|_2^4}\)-\frac1N\log\frac{|(m,v)|}{\|m\|^2_2}\leq \frac1N\log\(1+e^{N^{119/120}}\)\leq \frac{1}{N^{1/120}}.
\ee
Moreover for $m\in T$ Lemma \ref{lem:cond_law}, Lemma \ref{lemma:neperpnepara} and Proposition \ref{prop:logdetZfin} (with $\d'=1/120$) give

\bea
\frac1N\log\E\left|\det\nabla^2 F_{\TAP}(m)\big|_{\Omega(m)}\right|&=&\frac1N\log\E\det|P^{\perp}_{m}ZP^{\perp}_{m}+K(m)|\nn\\
&=&\frac1N\log\frac{|(m,v)|}{\|m\|^2_2}+\frac1N \log \E\left|\det \(Z^{(\parallel)}-\frac{\|x\|^2_2}{(m,v)}P_{e^1_{N-1}}\)\right|\nn\\
&=&\frac1N\log\frac{|(m,v)|}{\|m\|^2_2}+\int \log|x|\mu_m(dx)+r_N\nn\\
&=&\frac1N\log\(\frac{\|v\|^2_2}{\|m\|^2_2}+\frac{|(m,v)|}{\|m\|_2^2}-\frac{(m,v)^2}{\|m\|_2^4}\)+\int \log|x|\mu_m(dx)+r'_N\nn,
\eea
where 
\be
|r'_N|\lesssim \frac{1}{N^{1/120}}\,.
\ee

Next we focus on $m$ in the set
\be
T':=\{m\in L_{\a,N}\,:\, 0< |(m,v)|<e^{-N^{119/120}}\|v\|_2\}.
\ee
For $m\in T'$ we have
\be
\frac{|(m,v)|}{\|m\|^2_2}+\frac{\|v\|^2_2}{\|m\|_2^2}-\frac{(m,v)^2}{\|m\|_2^4}\leq \frac{|(m,v)|}{\|m\|^2_2}+\frac{\|v\|^2_2}{\|m\|_2^2}\leq \frac{\|v\|^2_2}{\|m\|_2^2}\(1+e^{-N^{119/120}}\), 
\ee
hence
\be\label{eq:thenwhat?4}
0\leq \frac1N\log\(\frac{|(m,v)|}{\|m\|^2_2}+\frac{\|v\|^2_2}{\|m\|_2^2}-\frac{(m,v)^2}{\|m\|_2^4}\)-\frac1N\log\frac{\|v\|^2_2}{\|m\|_2^2}\leq \frac1N\log\(1+e^{-N^{119/120}}\)\lesssim \frac{1}{N}.
\ee
By Lemma \ref{lem:cond_law} and Lemma \ref{lem:det_formula0} we have
\bea
\E[|\det \nabla^2 F_{\TAP}(m)\big|_{\Omega(m)}|]&=&\E\left[\left|\frac{(m,v)}{\|m\|_2^2}\det Z^{(\parallel)}-\frac{\|x\|_2^2}{\|m\|_2^2}\det Z^{(\perp)}\right|\right]\nn\\
&\leq&\frac{|(m,v)|}{\|m\|_2^2}\E[|\det Z^{(\parallel)}|]+\frac{\|x\|_2^2}{\|m\|_2^2}\E[|\det Z^{(\perp)}|]\nn\\
&\leq&e^{-N^{119/120}}\E[|\det Z^{(\parallel)}|]+\frac{\|x\|_2^2}{\|m\|_2^2}\E[|\det Z^{(\perp)}|]\nn\\
&=&\E[|\det Z^{(\perp)}|]\(e^{-N^{119/120}}\frac{\E[|\det Z^{(\parallel)}|]}{\E[|\det Z^{(\perp)}|]}+\frac{\|x\|_2^2}{\|m\|_2^2}\)\nn\\
&\leq&\E[|\det Z^{(\perp)}|]\(e^{N^{59/60}-N^{119/120}}+\frac{\|x\|_2^2}{\|m\|_2^2}\)\nn\\
&\leq&\E[|\det Z^{(\perp)}|]\(e^{-\frac{N^{119/120}}{100}}+\frac{\|x\|_2^2}{\|m\|_2^2}\)\nn\\
\eea
for $N$ large enough. Here we used Lemma \ref{lemma:verysleepy} in the penultimate inequality. Similarly
\bea
\E[|\det\nabla^2 F_{\TAP}(m)\big|_{\Omega(m)}|]&=&\E\left[\left|\frac{(m,v)}{\|m\|_2^2}\det Z^{(\parallel)}-\frac{\|x\|_2^2}{\|m\|_2^2}\det Z^{(\perp)}\right|\right]\nn\\
&\geq&\frac{\|x\|_2^2}{\|m\|_2^2}\E[|\det Z^{(\perp)}|]-\frac{|(m,v)|}{\|m\|_2^2}\E[|\det Z^{(\parallel)}|]\nn\\
&\geq& \E[|\det Z^{(\perp)}|]\(\frac{\|x\|_2^2}{\|m\|_2^2}-e^{-N^{119/120}}\frac{\E[|\det Z^{(\parallel)}|]}{\E[|\det Z^{(\perp)}|]}\)\nn\\
&\geq& \E[|\det Z^{(\perp)}|]\(\frac{\|x\|_2^2}{\|m\|_2^2}-e^{-\frac{N^{119/120}}{100}}\)\nn. 
\eea
again by Lemma \ref{lemma:verysleepy} and assuming $N$ sufficiently large. Therefore
\be\label{eq:thelast1}
\E[|\det\nabla^2 F_{\TAP}(m)\big|_{\Omega(m)}|]\leq \frac1N\log\E[|\det Z^{(\perp)}|]+\frac{1}{N}\log\(\frac{\|x\|_2^2}{\|m\|_2^2}\)+\frac{C}{N^{1/120}}
\ee
and
\be\label{eq:thelast2}
\E[|\det\nabla^2 F_{\TAP}(m)\big|_{\Omega(m)}|]\geq \frac1N\log\E[|\det Z^{(\perp)}|]+\frac{1}{N}\log\(\frac{\|x\|_2^2}{\|m\|_2^2}\)+\frac{C}{N^{1/120}},
\ee
where we used \eqref{eq:uselog} in either bounds. 
Combining (\ref{eq:logdetZfin2}), (\ref{eq:thelast1}), (\ref{eq:thelast2}), \eqref{eq:thenwhat?4} we recover (\ref{eq:NOTperppara}).
\end{proof}

Combining Lemma \ref{lemma:v=0}, Proposition \ref{prop:perppara}, Proposition \ref{prop:thelastone} we prove Theorem \ref{thm:MainDet2}.

To complete the proof of our main Theorem we need two additional result. The first one deals with those $m\neq0$ lying outside the set $L_{\a,N}$. The second one shows that the condition $v=0$ is satisfied by at most finitely many points in $(-1,1)^N$. 

\begin{proposition}\label{prop:outside}
Let $\a,\b>0$, $N\in \N$ large enough and consider $m\in(-1,1)^N\setminus\{0\}$ such that $m\notin L_{\a,N}$. 
It holds
\been{
\label{eq:outside}
\aled{
\frac1N\log\E[|\det\nabla^2 F_{\TAP}(m)\big|_{\Omega(m)}|]&\leq \frac1N\log\(\frac{\|v\|^2_2}{\|m\|^2_2}+\frac{|(m,v)|}{\|m\|^2_2}-\frac{|(m,v)|^2}{\|m\|^4_2}\)\\
&+2\log\(16\b(1+\b^2)\)-\frac{17}{N}\sum_{i\in[N]}\log(1-m^2_i).
}}
\end{proposition}
\begin{proof}
If $m\in L_{\a,N}$ Lemma \ref{lem:prop2} ensures that the matrix $D_N$ and the vector $v$ are bounded. For $m\in(-1,1)^N\setminus\{0\}$ such that $m\notin L_{\a,N}$ these quantities are no longer uniformly bounded and can be arbitrarily large as $m$ approaches the boundary of the hypercube. However Lemma \ref{prop:repr-det} and Lemma \ref{lem:det_formula0} are purely algebraic and the representation \eqref{eq:detY} still holds. Therefore
\been{
\label{eq:last_state1}
\aled{
\E[|\det\nabla^2 F_{\TAP}(m)\big|_{\Omega(m)}|]&\leq\frac{|(m,v)|}{\|m\|_2}\E\[|\det(Z^{(1,1)})|\]+\frac{\|x\|^2_2}{\|m\|^2_2}\E\[|\det(Z^{([2],[2])})|\]\\
&\leq \(\frac{|(m,v)|}{\|m\|_2}+\frac{\|x\|^2_2}{\|m\|^2_2}\)\(\E\[|\det(Z^{(1,1)})|\]+\E\[|\det(Z^{([2],[2])})|\]\right)\\
&\leq \(\frac{|(m,v)|}{\|m\|_2}+\frac{\|x\|^2_2}{\|m\|^2_2}\)\(\E\[|\det(Z^{(1,1)}+i)|\]+\E\[|\det(Z^{([2],[2])}+i)|\]\right),
}
}
where in the last line we used that $|\det(A)|\leq|\det(A+i)|$ for any real symmetric matrix $A$. 
The Weyl inequality yields
\been{
\label{eq:W_det}
|\l_i(Z^{([k],[k]})|\leq |\l_i((OD_N(m)O^T)^{([k],[k])})|+\frac{\b}{\sqrt{N}}\|J\|_{\textup{op}}, \quad i\in[N].
}
As before $\l_i(M)$ denotes the $i-$th eigenvalue of the matrix $M$ and $J$ denotes a $(N-1)\times (N-1)$ (for $k=1$) or $(N-2)\times (N-2)$ (for $k=2$) GOE matrix. From \eqref{eq:W_det} we get
\been{
\label{eq:last_state0}
\aled{
&\E\[|\det(Z^{([k],[k])}+i)|\]\leq \E\[\prod_{i\in [N]}\(\left|\l_i\((OD_N(m)O^T)^{([k],[k])}+i\)\right|+\frac{\b}{\sqrt{N}}\|J\|_{\textup{op}}\)\]\\
&\leq2^{N-1}\(\prod_{i\in [N]}\left|\l_i\((OD_N(m)O^T)^{([k],[k])}\)+i\right|+\frac{\b^N}{N^{\frac{N}{2}}}\E\[\|J\|^{N}_{\textup{op}}\]\)\\
&\leq2^{N-1}\(\left|\det\((OD_N(m)O^T)^{([k],[k])}+i\)\right|+\frac{\b^N}{N^{\frac{N}{2}}}\E\[\|J\|^{N}_{\textup{op}}\]\)\\
&\leq2^{N-1}\(|\det(D_N(m)+i)|e^{\left|\log\(\left|\det\(D_N(m)+iI_N\)\right|\)-\log\det\((OD_N(m)O^T)^{([k],[k])}+i\)\right|}+\frac{\b^N}{N^{\frac{N}{2}}}\E\[\|J\|^{N}_{\textup{op}}\]\),
}
}
where, in the second line, we used the inequality $(a+b)^N\leq 2^{N-1}a^N+2^{N-1}b^N$. 

We have
\been{
\label{eq:last_state4}
\aled{
&|\det(D_N(m)+i)|\leq \prod_{i\in [N]}\(\left|\frac{1}{1-m_i^2}+2\b^2(1-Q(m))+i\right|\)\\
&\leq \prod_{i\in [N]}\(\frac{1}{1-m_i^2}+2\b^2+1\)\\
&\leq \prod_{i\in [N]}\(\frac{2+2\b^2}{1-m_i^2}\)=e^{N\log(2+2\b^2)-\sum_{i\in[N]}\log(1-m^2_i)}.
}
}
Moreover by Lemma \ref{lemma:erstatz2} (with $A=D_N(m)+iI_N$)
\been{
\label{eq:last_state21}
\aled{
&\left| \log\(\left|\det\(D_N(m)+iI_N\)\right|\)- \log\(\left|\det\((OD_N(m)O^T)^{([k],[k])}+iI_{N-k}\)\right|\)\right|\\
 &\leq4k\log 2+8k|\log(|\|D_N(m)\|_{\textup{op}}+i|)|+4k|\log(\l_{\min}(D_N(m)+iI_N))|\\
 &+4k|\log(\l_{\min}((OD_N(m)O^T)^{([k],[k])}+iI_{N-k})|.
 }
 }
By the inequalities $1\leq |a+i|=|b+i|$ (for $0\leq a\leq b$) and $0\leq \l_{\min}(\,\cdot\,)\leq \|\,\cdot\,\|_{\textup{op}}$ we get
   \been{
   \label{eq:last_state22}
|\log(\l_{\min}(D_N(m)+iI_N))|)\leq  \log(|\|D_N(m)\|_{\textup{op}}+i|)
 }
 and
   \been{
    \label{eq:last_state23}
   \aled{
|\log(\l_{\min}((OD_N(m)O^T)^{([k],[k])}+iI_{N-k})|\leq  \log(|\|D_N(m)^{([k],[k])}\|_{\textup{op}}+i|)\leq  \log(|\|D_N(m)\|_{\textup{op}}+i|).
 }
 }
When we plug (\ref{eq:last_state22}), (\ref{eq:last_state23}) into (\ref{eq:last_state21}) we obtain
 \be\label{eq:last_state24}
 \left| \log\(\left|\det\(D_N(m)+iI_N\)\right|\)- \log\(\left|\det\((OD_N(m)O^T)^{([k],[k])}+iI_{N-k}\)\right|\)\right|\leq 4k\log 2+16k |\log|\|D_N(m)\|_{\textup{op}}+i||.
 \ee
The r.h.s of the display above can be bounded as follows:
\been{
\label{eq:last_state24}
\aled{
&|\log(|\|D_N(m)\|_{\textup{op}}+i|)|=\max_{i\in [N]}\left|\log\(\left|\frac{1}{1-m_i^2}+2\b^2(1-Q(m))+i\right|\)\right|\\
&\leq\max_{i\in[N]}\log\(\frac{1}{1-m^2_i}+2\b^2+1\)\\
&\leq\log(2+2\b^2)-\min_{i\in[N]}\log(1-m^2_i)\leq \log(2+2\b^2)-\sum_{i\in[N]}\log(1-m^2_i),
 }
 }
where in the second line we used $|a+i b|\leq |a|+|b|$ and $2\b^2(1-Q(m))\leq 2\b^2$.
Hence
\been{
\label{eq:last_state3}
\aled{
&\left| \log\(\left|\det\(D_N(m)+i\)\right|\)- \log\(\left|\det\((OD_N(m)O^T)^{([k],[k])}+i\)\right|\)\right|\\
 &\leq4k\log 2+16k\log(2+2\b^2)-16k\sum_{i\in[N]}\log(1-m^2_i).
}
}

Finally, using that for all $t>0$ 
\been{
P\(\frac{\b}{\sqrt{N}}\|J\|_{\textup{op}}\geq \sqrt{2}\b+t\)\leq e^{-\frac{t^2N}{2\b^2}}.
}
we bound for $N$ sufficiently large
\been{
\label{eq:last_state2}
\aled{
\frac{\b^N}{N^{\frac{N}{2}}}\E\[\|J\|^{N}_{\textup{op}}\]&=N\int_0^\infty dt\, t^{N-1}P\(\frac{\b}{\sqrt{N}}\|J\|_{\textup{op}}\geq t\)\\
&\leq N(\sqrt2\b)^N+N\int_0^\infty dt (t+\sqrt2\b)^{N-1}P\(\frac{\b}{\sqrt{N}}\|J\|_{\textup{op}}\geq \sqrt{2}\b+t\)\\
&\leq N(\sqrt2\b)^N+N\int_0^\infty dt e^{-N\(\frac{t^2}{2\b^2}-\log\(t+\sqrt2\b\)\)}\\
&\lesssim N(\sqrt2\b)^N\leq (2\b)^N. 
}}
Now we plug the bounds \eqref{eq:last_state4}, \eqref{eq:last_state3} and \eqref{eq:last_state2} into \eqref{eq:last_state0}. We obtain
\bea
\frac12\E\[|\det(Z^{([1],[1])}+i)|\]&\leq& e^{4\log 2+16\log(2+2\b^2)+N\log(4+4\b^2)-17\sum_{i\in[N]}\log(1-m^2_i)}+(4\b)^N,\nn\\
&\leq&\frac12e^{2N\log\(16\b(1+\b^2)\)-17\sum_{i\in[N]}\log(1-m^2_i)}\label{eq:terr1}\\
\frac12\E\[|\det(Z^{([2],[2])}+i)|\]&\leq& e^{8\log 2+32\log(2+2\b^2)+N\log(4+4\b^2)-33\sum_{i\in[N]}\log(1-m^2_i)}+(4\b)^N\nn\\
&\leq&\frac12e^{2N\log\(16\b(1+\b^2)\)-33\sum_{i\in[N]}\log(1-m^2_i)}\label{eq:terr2}.
\eea
Here we used again that $N$ is large enough and also $e^a+e^b\leq 2e^{|a|+|b|}$. Thus by \eqref{eq:last_state1}
\be
\E[|\det\nabla^2 F_{\TAP}(m)\big|_{\Omega(m)}|]\leq \(\frac{|(m,v)|}{\|m\|_2}+\frac{\|x\|^2_2}{\|m\|^2_2}\)e^{2N\log\(16\b(1+\b^2)\)-17\sum_{i\in[N]}\log(1-m^2_i)},
\ee
and the proof is complete.
\end{proof}

\begin{lemma}\label{lemma:finitelymany}
Let $\beta>0,h\in\R$. The set $\{m\in(-1,1)^N\,:\,v=0\}$ is discrete with
\be
\card \{m\in(-1,1)^N\,:\,v=0\}\lesssim 3^N. 
\ee
\end{lemma}

\begin{proof}
Recalling (\ref{eq:defD2}) and (\ref{eq:defu-v-sect5}), $v=0$ is equivalent to
\be\label{eq:=Q}
\frac{m}{1-m^2}+h-\atanh m=4\b^2mQ(m)
\ee
where we use the notation for which the vector $f(m)$ has $i$-th component $f(m_i)$. Note that if $h=0$ then $m=0$ is a solution of \eqref{eq:=Q}, if $h\neq0$ it is never a solution. In either cases we can and will exclude $m=0$ for our next considerations. 

For $m\in(-1,1)\setminus\{0\}$ we set
$$
U(m):=\frac{1}{4\beta^{2}}\(\frac{1}{1-m^{2}}-\frac{\text{arctanh}(m)-h}{m}\). 
$$
Thus
\begin{align}
\{m\in(-1,1)^N\setminus\{0\}\mid v=0\}&=\{m\in(-1,1)^N\setminus\{0\}\mid U(m_i)=Q(m)\, \forall i\in[N]\}\nn\\
&=\bigcup_{q\in(0,1]} \left\{m\in(-1,1)^N\setminus\{0\}\mid Q(m)=q,\,\,U(m_i)=q\, \forall i\in[N]\right\}\,. \label{eq:2cups}
\end{align}
We first argue that for a fixed $q\in[0,1]$ 
\begin{equation}\label{eq:one_dim_Eq-1}
U(m)=q,\quad m\in(-1,1)\setminus\{0\},
\end{equation}
has finitely many solutions,
and that these solutions are real analytic as functions of $q$. Indeed
one can check that $U'$ has a single root $\hat{m}\in[0,1)$ so that
the restricted functions $U_{1}:=U|_{(-1,0)},U_{2}:=U_{2}|_{(0,\hat{m})},U_{3}:=U|_{(\hat{m},1)}$
are analytic with non-vanishing derivative (where $\hat{m}=0$ iff
$h=0$ and we define $U_{2}$ only for $h\ne0$). The inverses $\widetilde m_k:=U_{k}^{-1}$, $k=1,2,3$,
thus exist and are analytic functions. Setting $q_1:=0$ and $q_2,q_3:=U(\hat{m})$, we have $\widetilde m_{k}:[q_{k},\infty)\to(-1,1)$.
Thus the set of solutions of (\ref{eq:one_dim_Eq-1}) at given $q
\in[0,1]$
is 
\be
\mathcal M(q):=\bigcup_{k\in\left\{ 1,2,3\right\} :q\ge q_{k}}\{\widetilde m_{k}(q)\}.
\ee

When solving the vector equation for a fixed $q\in[0,1]$
\begin{equation}\label{eq:one_dim_Eq-1V}
U(m)=q,\quad m\in(-1,1)^N\setminus\{0\},
\end{equation}
we can pick for each coordinate one element of $\mathcal M(q)$, that is if $m$ is a solution of \ref{eq:one_dim_Eq-1V} then for every $i\in[N]$ there is $K_i\in\{1,2,3\}$ such that $m_i=\widetilde m_{K_i}(q)$.  We set $q_{K}=\max_{i\in[N]}q_{K_{i}}$ and denote the frequency of each $K_i$ by $\alpha_{k}:=\frac{\card\left\{ i:K_{i}=k\right\} }{N}$.

Let now
$g_{K}:[q_{K},\infty)\to\mathbb{R}$ be defined by 
\be
g_{K}(q):=\sum_{k=1}^{4}\alpha_{k}\widetilde m_{k}(q)^{2}-q.
\ee
The function $g_{K}$ is analytic on $(q_{K},\infty)$, so if it has
infinitely many roots in $[q_{K},1)$ then in fact $g_{K}=0$ on $[q_{K},\infty)$
which is impossible, since the first term of $g_{K}(q)$
is bounded as $q\to\infty$. Thus $\{q\in[q_K,1)\mid g_{K}=0\}$ has at most finitely many
elements (uniformly in $N$). 

Therefore 
\begin{align}
\eqref{eq:2cups}=\bigcup_{\{q\in[q_K,1)\mid g_{K}=0\}}\bigcup_{\{K_1,\ldots,K_N\}\in[3]^N}\{(\widetilde m_{K_1}(q),\ldots,\widetilde m_{K_N}(q))\},
\end{align}
i.e. the set of $m$ such that $v=0$ is discrete and it contains at most a number of points proportional to $3^N$.
\end{proof}

Combining Theorem \ref{thm:MainDet2} with Proposition \ref{prop:outside} and Lemma \ref{lemma:finitelymany} gives Theorem \ref{thm:MainDet}.

%



\bibliographystyle{amsplain}
\bibliography{bibdet}
\end{document}